\documentclass[11pt,english,a4paper]{smfart}
\usepackage[english]{babel}
\usepackage{amssymb,xspace}
\usepackage{amstext}
\usepackage{smfthm}
\theoremstyle{plain}
\usepackage{amsbsy,amssymb,amsfonts,latexsym}

\usepackage{enumerate}
\usepackage{graphics}
\usepackage{pstricks}
\usepackage{pst-node}

\title{On Read's type operators on Hilbert spaces}

\author{Sophie Grivaux}
\address{
Laboratoire Paul Painlev\' e, UMR 8524, Universit\'e des Sciences et
Technologies de Lille, Cit\' e Scientifique, 59655 Villeneuve d'Ascq
Cedex, France}
\email{grivaux@math.univ-lille1.fr}

\author{Maria Roginskaya}
\address{Department of Mathematical Sciences,
Chalmers University of Technology, SE-41296 G\"oteborg, Sweden, \emph{and}
Department of Mathematical Sciences,
G\"oteborg University, SE-41296 G\"oteborg, Sweden}
\email{maria@chalmers.se}


\thanks{The first author was partially supported
by ANR-Projet Blanc DYNOP}

\newcommand{\ch }{
\newdimen\hauteur \hauteur=7.3 pt
\mkern 3 mu\hbox
{\vrule height \hauteur depth 0 pt width .4 pt\kern 0.8 pt
\vrule height \hauteur depth 0 pt width .4 pt\kern -2.51 pt
\vrule height .4 pt depth 0 pt width 3.3 pt}
\raise 7 pt \hbox{\kern -2.4 pt
\vrule height .4 pt depth 0 pt width 1.2 pt}
\raise 5.5 pt \hbox{\kern -2.5 pt
\vrule height .2 pt depth 0 pt width .56 pt}
\raise 5.6 pt \hbox{\kern -.46 pt
\vrule height .2 pt depth 0 pt width .56 pt}
\raise 5.7 pt \hbox{\kern -.46 pt
\vrule height .2 pt depth 0 pt width .56 pt}
\raise 5.8 pt \hbox{\kern -.46 pt
\vrule height .2 pt depth 0 pt width .56 pt}
\raise 5.9 pt \hbox{\kern -.46 pt
\vrule height .2 pt depth 0 pt width .56 pt}
\raise 6 pt \hbox{\kern -.46 pt
\vrule height .2 pt depth 0 pt width .56 pt}
\raise 6.1 pt \hbox{\kern -.46 pt
\vrule height .2 pt depth 0 pt width .56 pt}
\raise 6.2 pt \hbox{\kern -.46 pt
\vrule height .2 pt depth 0 pt width .56 pt}
\raise 6.3 pt \hbox{\kern -.46 pt
\vrule height .2 pt depth 0 pt width .56 pt}
\raise 6.4 pt \hbox{\kern -.46 pt
\vrule height .2 pt depth 0 pt width .56 pt}
\raise 6.5 pt \hbox{\kern -.46 pt
\vrule height .2 pt depth 0 pt width .56 pt}
\raise 6.6 pt \hbox{\kern -.46 pt
\vrule height .2 pt depth 0 pt width .56 pt}
\kern 2 pt}

\newcommand{\ls}{\lesssim}
\newcommand{\gs}{\gtrsim}

\def\K{\ensuremath{\mathbb K}}

\def\R{\ensuremath{\mathbb R}}
\def\Z{\ensuremath{\mathbb Z}}

\def\C{\ensuremath{\mathbb C}}

\def\N{\ensuremath{\mathbb N}}
\def\P{\ensuremath{\mathbb P}}

\newcommand{\sep}{separable}

\newcommand{\hy}{hypercyclic}

\newcommand{\ops}{operators}
\newcommand{\op}{operator}

\newcommand{\wrt}{with respect to}

\newcommand{\proba}{probability}
\newcommand{\ga}{Gaussian}

\newcommand{\mea}{measure}
\newcommand{\nd}{non-degenerate}

\newcommand{\ifff}{if and only if}

\newcommand{\pss}[2]{\ensuremath{{\langle #1,#2\rangle}}}

\newcommand{\loi}{lay-off interval}
\newtheorem{theorem}{Theorem}[section]

\newtheorem{lemma}[theorem]{Lemma}

\newtheorem{proposition}[theorem]{Proposition}

\newtheorem{corollary}[theorem]{Corollary}

{\theoremstyle{definition}}

{\theoremstyle{definition}}

\newtheorem{fact}[theorem]{Fact}

{\theoremstyle{definition}}

{\theoremstyle{definition}}

{\theoremstyle{definition}\newtheorem{remark}[theorem]{Remark}}

{\theoremstyle{definition}\newtheorem*{FFC Criterion}{Frequent
Faber-hypercyclicity Criterion}}
\newtheorem*{Hypercyclicity Criterion}{Hypercyclicity Criterion}
{\theoremstyle{definition}\newtheorem*{GS Criterion}{Godefroy-Shapiro
Criterion}}
\def\piednote#1{\let\oldfn=\thefootnote\def\thefootnote{}\footnote{\noindent#1}%
\addtocounter{footnote}{-1}\def\thefootnote{\oldfn}}

\begin{document}

\begin{abstract}
Using Read's construction of operators without non-trivial invariant
subspaces/subsets on $\ell_{1}$ or $c_{0}$, we construct examples of
operators on a Hilbert space whose set of hypercyclic vectors is
``large'' in various senses. We give an example of an operator such that
the closure of
every orbit is a closed subspace, and then, answering a question of D.
Preiss, an example of an operator such that the  set of
its non-hypercyclic vectors is Gauss null. This operator has the property that it
is orbit-unicellular, i.e. the family of the closures of its orbits is
totally ordered. We also exhibit an example of an
operator on a Hilbert space which is not orbit-reflexive.
\end{abstract}
\maketitle

\section{Introduction}
Let $X$ be a real or complex infinite-dimensional \sep\ Banach
 space, and $T$ a bounded \op\ on $X$. In this paper we will be concerned
 with the study of the structure of
 orbits of vectors $x\in X$ under
 the action of $T$ from various points of view. If $x$ is any vector
 of $X$, the \emph{orbit} of $x$ under $T$ is the set
$\mathcal{O}rb(x,T)=\{T^{n}x \textrm{ ; } n\geq 0\}$.
The closure of this orbit is denoted by $\overline{\mathcal{O}rb}(x,T)$.
 The \emph{linear
orbit} of $x$ is the linear span of the orbit of $X$, i.e. the set
$\{p(T)x \textrm{ ; } p\in \K[\zeta ]\}$, $\K=\R$ or $\C$. When the
linear orbit of $x$ is dense, $x$ is said to be \emph{cyclic}, and $x$ is
said to be
\emph{\hy}\ when the orbit itself is dense. An \op\ admitting a
cyclic (resp. \hy)\ vector is called cyclic (resp \hy).
\par\smallskip
The structure of the set $HC(T)$ of \hy\ vectors for a \hy\ \op\ $T\in
\mathcal{B}(X)$ has been the subject of many investigations: linear
structure ($HC(T)$ always contains a dense linear manifold, see \cite{Bo},
sometimes an
infinite-dimensional closed subspace, see \cite{LM}), topological structure ($HC(T)$ is
a dense $G_{\delta }$ subset of $X$), measure-theoretic structure (see
for instance \cite{F}, \cite{BG})... In
particular, it is interesting to look for \ops\ whose set of \hy\ (or
even cyclic) vectors is as large as possible, especially in the Hilbert
space setting. Throughout the paper $H$ will denote a real or complex
separable infinite-dimensional Hilbert space. A major open question in
\op\ theory is to know whether, given any bounded
\op\ $T$ on $H$, there exists a closed subspace $M$ (resp. a closed
subset $F$) which is non-trivial, i.e. $M\not = \{0\}$  and $M\not = H$,
and invariant by $T$, i.e. $T(M)\subseteq M$ (resp, with $F$). These
problems are known as the Invariant Subspace and the Invariant Subset
Problems. If one does not work with \ops\ acting on a Hilbert space, but
with \ops\ acting on general \sep\ Banach spaces instead, the question
has been answered in the negative by Enflo \cite{E} and Read \cite{R}.
Read in particular constructed an \op\ without non-trivial invariant
subspaces in the space $\ell_{1}$ of summable sequences, and even an \op\
without non-trivial invariant closed subsets on $\ell_{1}$ \cite{R2}. In
other words $HC(T)=\ell_1\setminus \{0\}$ for this \op.
 The Invariant Subspace Problem is still open
 in the reflexive setting, and the closest one could get
 \cite{R3} to this
 are examples of operators without non-trivial
invariant subspaces on some spaces with separable dual, such as $c_{0}$
for instance.
\par\smallskip
Our aim in this paper is to present a simplified version of Read's
construction in \cite{R2} which is adapted to the Hilbert space setting,
and to obtain in this way \ops\ whose orbits have interesting properties:
we first construct an example of a Hilbert space \op\  such that the
orbit of every vector $x$ coincides with its linear orbit. This
corresponds to the construction of what we call the ``(c)-part'' in
Read's type
\ops\ (see Section $2$ for definitions).

\begin{theorem}\label{th1}
There exists a \hy\
 \op\ on $H$ such that for every vector $x\in H$, the closure of
 the orbit
$\mathcal{O}rb(x,T)$ is a subspace, i.e. the closures of the two sets
$\{T^{n}x \textrm{ ; } n\geq 0\}$ and $\{p(T)x \textrm{ ; } p\in \K[\zeta
]\}$ coincide.
\end{theorem}

We define in Section $2$ the \ops\ which will be needed for the proof of
Theorem \ref{th1}, explain the role of the (c)-fan, and then prove
Theorem \ref{th1} in Section $3$. This can be seen as the basic
construction, and in Section $4$ we elaborate on it to prove the next
results.
\par\smallskip
Section $4$ is devoted to the study of the set $HC(T)$ from the point of
view of geometric measure theory: it is well known and easy to prove that
whenever $T$ is hypercyclic on $X$, $HC(T)$ is a dense $G_{\delta }$
subset of $X$, or equivalently, its complement $HC(T)^{c}$ is
a set of the first category,
i.e. a countable union of closed sets with empty interior,
so $HC(T)^{c}$ is a ``small'' set from this point of view. Increasing the
size of $HC(T)$ means having $HC(T)^{c}$ smaller, and various notions of
smallness have been considered in this setting. In particular, Bayart
studied in \cite{B} examples of \ops\ such that $HC(T)^{c}$ was $\sigma
$-porous, i.e. a countable union of porous sets. The notion of porosity
quantifies the fact that a set has empty interior: a subset $E$ of a
Banach space $X$ is called \emph{porous} if there exists a $\lambda \in
]0,1[$ such that the following is true: for every $x\in E$  and every
$\varepsilon >0$, there exists a point $y\in X$ such that
$0<||y-x||<\varepsilon $ and  $E\cap B(y, \lambda ||y-x||)$ is empty. A
countable union  of porous sets is called \emph{$\sigma $-porous}. We
refer the reader to the references \cite{Z} or \cite{BL} for more
information on porous and $\sigma $-porous sets, and their role in
questions related to the differentiation of Banach-valued functions.
\par\smallskip
Bayart constructed in \cite{B} examples of \ops\ $T$ on $F$-spaces such
that $HC(T)^{c}$ was $\sigma $-porous, but on Banach spaces the only
\ops\  which were known to have this property were the ones without
nontrivial closed invariant subsets. Hence a question of \cite{B} was to
know whether it was possible to have a Hilbert space \op\ $T$ such that
$HC(T)^{c}$ was $\sigma $-porous. This question was answered in the
affirmative by David Preiss \cite{P}, who constructed a  bilateral
weighted shift on $\ell_{2}(\Z)$ having this property. This proof has
been since recorded in \cite{BMM}. This leads to another question, which
was asked by David Preiss too \cite{P}: does there exist a Hilbert space
\op\ such that $HC(T)^{c}$ is Haar null? Recall that a subset $A$ of $H$
is said to be \emph{Haar null} if there exists a Borel
 \proba\ \mea\ $m$ on
$H$ such that for every $x\in H$, the translate $x+A$ of $A$ has
$m$-\mea\ $0$. The class of Haar null sets is another $\sigma $-ideal of
``small sets'', different from the class of $\sigma $-porous sets, and
actually these two classes are not comparable: a result of Preiss and
Tiser \cite{PT} is that any (real) Banach space can be decomposed as the
disjoint union of two sets, of which one is $\sigma $-porous and the other
Haar null. See \cite{BL} for more on this and related classes of
negligible sets.
\par\smallskip
We answer here Preiss's question in the affirmative by showing the
following stronger result:

\begin{theorem}\label{th2}
There exists a bounded \op\ $T$ on the Hilbert space $H$ such that the set
$HC(T)^{c}$ is a countable union of subsets of closed hyperplanes of $H$.
In particular $HC(T)^{c}$ is Gauss null (hence Haar null) and $\sigma
$-porous.
\end{theorem}

Recall that a subset $A$ of $H$ is said to be \emph{Gauss null} if for
every \nd\ \ga\ \mea\ $\mu $ on $H$, $\mu (A)=0$. Since the $\mu $-\mea\
of a closed hyperplane vanishes for every such $\mu $, $HC(T)^{c}$ will
clearly be Gauss null.
\par\smallskip
The proof of Theorem \ref{th2} requires that we complicate a bit the
construction of Section $2$, and we introduce what we call the
``(b)-part'' in Read's examples in order to achieve this. For clarity's
sake we show first in Section $4$ that an \op\ $T$ can be constructed
with $HC(T)^{c}$ Haar null (and $\sigma $-porous). Then we show in
Section $5$ the following result, which is interesting in itself and
which easily implies Theorem \ref{th2}:

\begin{theorem}\label{th2bis}
There exists a bounded \op\ $T$ on $H$ which is
\emph{orbit-unicellular}: the family $(
\overline{\mathcal{O}rb}(x,T))_{x\in H}$ of all the closures of its
orbits is totally ordered, i.e. for any pair $(x,y)$ of vectors of $H$,
either $\overline{\mathcal{O}rb}(x,T)\subseteq
\overline{\mathcal{O}rb}(y,T)$ or $\overline{\mathcal{O}rb}(y,T)
\subseteq \overline{\mathcal{O}rb}(x,T)$. In particular the \op\ induced
by $T$ on any invariant subspace $M$ of $H$ is \hy, i.e. $M=\overline{\mathcal{O}rb}
(x,T)$ for some $x\in H$.
\end{theorem}

The term ``orbit-unicellularity'' comes from the fact that an
\op\ is said to be \emph{unicellular} if the lattice of its invariant
subspaces is totally ordered. When an operator $T$ is unicellular, every
invariant subspace $M$ of $T$ is cyclic, i.e.
is the closure $M_{x}$ of the linear orbit of some vector $x\in H$, and the
unicellularity of $T$ is equivalent to the fact that for every pair $(x,y)$
of vectors of $H$, either $M_{x}\subseteq M_{y}$ or $M_{y}\subseteq
M_{x}$. See for instance \cite{RR} for some examples of unicellular
\ops. In our case $\overline{\mathcal{O}rb}(x,T)=M_{x}$, so $T$ is in
particular unicellular. Let us underline here that the point of Theorem
\ref{th2bis} is that we are dealing with \hy\ vectors, and not with cyclic
ones: of course there are many \ops\ whose lattice of invariant
subspaces it totally ordered. This is the case for the Volterra \op\ $V$
on $L^{2}([0,1])$ for instance: each invariant subspace for $V$ is of
the form $M_{t}=\{f\in L^{2}([0,1]) \textrm{ ; } f=0 \textrm{ a.e. on } (0,t)
\}$,
$t\in [0,1]$. In this case the lattice of invariant subspaces is
isomorphic to $\R$ with its natural order. It is even possible that the
lattice of invariant subspaces be countable: this is the case for instance for some
weighted unilateral backward shifts on $\ell_{2}(\N)$, the Donoghue
\ops. Here the non trivial invariant subspaces are exactly the finite
dimensional spaces $M_{n}=\textrm{sp}[e_{0},\ldots, e_{n}]$, $n\geq 0$,
where $(e_{n})_{n\geq 0}$ is the canonical basis of $\ell_{2}(\N)$. It
is worth noting that such a situation cannot occur for an \op\ whose
closure of orbits are subspaces and which is orbit-unicellular. Indeed
suppose that $M$ and $N$ are two invariant subspaces for $T$ with $N\subsetneq
M$. As was mentioned in Theorem \ref{th2bis}, there exist two vectors $x$
and $y$ such that $M=\overline{\mathcal{O}rb}(x,T)$ and
$N=\overline{\mathcal{O}rb}(y,T)$. It is easy to see that the \op\
induced by $T$ on the quotient $M/N$ is \hy, which implies that $M/N$ is
infinite-dimensional. Hence the ``gap'' between two invariant subspaces
of $T$, if non trivial, is of infinite dimension. This leads to the
following observation:

\begin{proposition}\label{prop2ter}
The following dichotomy holds true:
\begin{enumerate}
\item[(a)] either there exists a bounded \op\  $T$ on an
infinite-dimensional separable Hilbert space which has no non trivial
invariant closed subset;

\item[(b)] or every \op\ acting on an
infinite-dimensional separable Hilbert space, whose
closure of orbits are subspaces and which is orbit-unicellular, has the
following
property: there exists a family of (closures of) orbits which is
order isomorphic to $(\R, \leq)$. In particular such an \op\ has
uncountably many distinct (closures of) orbits.
\end{enumerate}
\end{proposition}
\par\smallskip
In Section $6$ we give a positive answer to a question of \cite{HRR}
which concerns orbit-reflexive \ops\ on Hilbert spaces. If $T\in
\mathcal{B}(X)$ is a bounded \op\ on $X$, $T$ is said to be
\emph{orbit-reflexive} if whenever $A\in \mathcal{B}(X)$ is such that
$Ax$ belongs to the closure of $\mathcal{O}rb(x,T)$ for every $x\in X$,
then $A$ must belong to the closure of the set $\{T^{n}\textrm{ ; } n
\geq 0\}$ for the Strong Operator Topology (SOT). In particular,
$A$ and $T$ must commute. Various conditions are given in \cite{HRR}
under which an \op\ on a Hilbert space is orbit-reflexive: for instance
any contraction on a Hilbert space is orbit-reflexive. The following
question is asked in \cite{HRR}: does there exist an \op\ on a Hilbert
space which is not orbit-reflexive? This question was pointed out to us
by Vladimir M\"uller \cite{M}. We answer it here in the affirmative:

\begin{theorem}\label{th3}
There exists a bounded \op\ on a Hilbert space which is not
orbit-reflexive.
\end{theorem}

Theorem \ref{th3} follows from a slight modification of the construction
of Section $4$. After this paper was submitted for publication, we were
informed  by Vladimir M\"uller that a much more simple example of
a non orbit-reflexive Hilbert
space \op\ was constructed independently in
\cite{MV}, as well as an example of an \op\ on the space $\ell_{1}(\N)$
which is reflexive but not orbit-reflexive.
\par\smallskip
We finish this introduction with a comment: we have mentioned previously
that the proofs of Theorems \ref{th1}, \ref{th2} and \ref{th3} involve
\ops\ of Read's type, and we use the (c)-part and the (b)-part of it. The
reader may justly ask about a possible (a)-part: such an (a)-part indeed
appears in Read's constructions in \cite{R}, \cite{R2} or \cite{R3}, and
it is actually the part which provides the vectors which belong to the
closures of all the sets $\overline{\mathcal{O}rb}(x,T)$.

\section{Making all orbits into subspaces: the role of the (c)-fan}
We start from the Hilbert space $\ell_{2}(\N)$ of square-summable
sequences indexed by the set $\N$ of nonnegative integers, with its
canonical basis $(e_{j})_{j\geq 0}$. A vector $x$ of $\ell_{2}(\N)$ is as
usual said to be \emph{finitely supported} if all but finitely many of
its coordinates on the basis $(e_{j})_{j\geq 0}$ vanish, and the set of
finitely supported vectors will be denoted by $c_{00}$. The forward shift
$T$ on $\ell_{2}(\N)$ is the \op\ defined by $Te_{j}=e_{j+1}$ for every
$j\geq 0$.
\par\smallskip
If $(f_{j})_{j\geq 0}$ is a sequence of finitely supported vectors such
that $f_{0}=e_{0}$ and $\textrm{sp}[f_{0}, \ldots,
f_{j}]=\textrm{sp}[e_{0},\ldots, e_{j}]$ for every $j\geq 1$ ($f_{j}$
belongs to $\textrm{sp}[e_{0},\ldots, e_{j}]$ and the $j^{th}$ coordinate
of $f_{j}$ on the basis $(e_{j})_{j\geq 0}$ is non-zero), then one can
define on $c_{00}$ a new norm associated to the sequence $(f_{j})_{j \geq
0}$. For any finite subset $J$ of $\N$ and any collection $(x_{j})_{j\in
J}$ of scalars,
$$||\sum_{j\in J}
x_{j}f_{j}||=\left(\sum_{j\in J}|x_{j}|^{2}\right)^{\frac {1}{2}}.$$ The
completion of $c_{00}$ under this
 new norm is a Hilbert space, with
the sequence $(f_{j})_{j\geq 0}$ as an orthonormal basis. We are going to
show that for a suitable choice of the sequence $(f_{j})_{j\geq 0}$, the
\op\ $T$ acting on $c_{00}$ extends to a bounded \op\ on the Hilbert
space $H:=H_{(f_{j})}$ which satisfies the properties of Theorem
\ref{th1}.
\par\smallskip
We denote by $ \K[\zeta ]$ the space of polynomials with coefficients in
$\K=\R$ or $\C$, and by $\K_{d}[\zeta ]$ the space of polynomials of
degree at most $d$. For $p\in\K[\zeta ]$, $p(\zeta
)=\sum_{k=0}^{d}a_{k}\zeta
^{k}$, we write as usual $|p|=\sum_{k=0}^{d}|a_{k}|$.
Let $(d_{n})_{n\geq 1}$ be an increasing sequence of positive integers,
and for every $n\geq 1$ let $(p_{k,n})_{1\leq k\leq k_{n}}$ be a finite
family of polynomials of degree at most $d_{n}$ with $|p_{k,n}|\leq 2$
for every $1\leq k\leq k_{n}$. In the proofs of the theorems, the
polynomials $p_{k,n}$ will have to satisfy some additional properties,
the most usual one being that the family $(p_{k,n})_{1\leq k\leq k_{n}}$
forms an $\varepsilon
_{n}$-net of the ball of radius $2$ of $\K_{d_{n}}[\zeta ]$, but
since these families will be chosen differently in the proofs of the
four theorems, we present for the time being the general construction.
\par\smallskip
The construction of the vectors $f_{j}$, $j\geq 0$, is to be done by
induction, starting from $f_{0}=e_{0}$. At step $n$, vectors $f_{j}$
 will be constructed for $j\in
[\xi _{n}+1,\xi_{n+1}]$, where $(\xi _{n})_{n\geq 0}$ is a sequence with
$\xi _{0}=0$ which will be chosen to grow very fast. We emphasize that
all the constants we are going to construct at step $n$ are determined
by the various constants which are constructed through steps $0$ to
$n-1$. When we say that a certain constant $C_{\xi _{n}}$
depends only on $\xi _{n}$, it means that it depends only on the
construction from steps $0$ to $n-1$. The construction is done by induction on $n$, and in all our
statements we assume that the construction has been carried out until
step $n-1$.
\par\smallskip
There will be two different types of definitions of $f_{j}$ for $j\in
[\xi _{n}+1,\xi _{n+1}]$, depending on whether $j$ belongs or not to a
collection of intervals called the \emph{fan} (we will later on call it the
\emph{(c)-fan}, to distinguish it from another fan which is going to be
introduced afterwards): this fan is a
lattice of intervals which we call
\emph{working
intervals}, and their role is to ensure that every orbit is a linear
manifold. The intervals between working intervals we call \emph{\loi s}:
on a \loi, $f_{j}$ is defined as
 $f_{j}=\lambda _{j}e_{j}$, where $\lambda _{j}$
is a scalar coefficient which is very large if $j$ belongs to the
beginning of the \loi, and very small if $j$ belongs to its end, while
the quotient $(\lambda _{j}/\lambda _{j+1} )$ is very close to $1$.  Thus
when both $j$ and $j+1$ belong to a \loi, $Tf_{j}=\lambda
_{j}e_{j+1}=(\lambda _{j}/\lambda _{j+1} )f_{j+1}$ and $T$ acts as a
weighted shift. So in a sense, ``nothing much happens on the \loi s'',
which explains their name. Their role is to prevent ``side effects'' from
the working intervals, which do the real work.
\par\smallskip
Here are now the precise definition of the vectors $f_{j}$, $j=\xi
_{n}+1, \ldots, \xi _{n+1}$.
 For any finite
sub-interval $A$ of $\N$, we denote by $\pi_{A}$ the projection of
$c_{00}$ onto the span of the vectors $f_{j}$, $j\in A$. Since we will
always require that $\textrm{sp}[f_{0}, \ldots,
f_{j}]=\textrm{sp}[e_{0},\ldots, e_{j}]$, $x$ belongs to $c_{00}$ \ifff\
it is finitely supported in $H$
\wrt\ $(f_{j})_{j\geq 0}$. When we talk of support in the sequel, we
will always mean \wrt\ $(f_{j})_{j\geq 0}$: $x$ is supported in $A$ if
$x=\sum_{j\in A} x_{j}f_{j}$. The norm $||\, .\, ||$ is the norm of $H$.

\subsection{Construction of the fan}
Let $c_{1,n}<c_{2,n}<\dots<c_{k_{n},n}$ be an extremely fast increasing
sequence of integers with $c_{1,n}$ very large \wrt\ $\xi
_{n}$. The fan consists of the lattice of all the intervals
$$I_{r_{1},r_{2}, \dots, r_{k_{n}}}=[r_{1}c_{1,n}+r_{2}
c_{2,n}+\dots+r_{k_{n}}c_{k_{n},n}, r_{1}c_{1,n}+r_{2}
c_{2,n}+\dots+r_{k_{n}}c_{k_{n},n}+\xi _{n}],$$ where $r_{1}, \dots,
r_{k_{n}}$ are nonnegative integers belonging to $[0,h_{n}]$. Here
$h_{n}$ is a very large integer depending only on $\xi _{n}$, but not on
the $c_{k,n}$'s, which will be chosen later on in the proof. If the gaps
between the different $c_{k,n}$'s are large enough, all these
$k_{n}h_{n}$ intervals are disjoint. For $k\in[1,k_{n}]$, we call $r_{k}$
the $k^{th}$ coordinate of the interval $I_{r_{1},r_{2}, \dots,
r_{k_{n}}}$, and write $|r|=r_{1}+\dots+r_{k_{n}}$.
\par\smallskip
Let $t\in [1,k_{n}]$ be the largest integer such that $r_{t}\geq 1$. We
will write $I_{r_{1},r_{2}, \dots, r_{k_{n}}}=I_{r_{1},r_{2}, \dots,
r_{t}}$ when there is no risk of confusion. For $j\in I_{r_{1},r_{2},
\dots, r_{t}}$, we define $f_{j}$ to be
$$f_{j}=\frac {1}{\gamma _{n}}4^{1-|r|}(e_{j}-p_{t,n}(T)e_{j-c_{t,n}}),$$
where $\gamma _{n}$ is a very small positive number depending only on
$\xi _{n}$ which will be chosen in the sequel. The interest of this
definition is twofold: first of all, we can already justify the name of
working interval, simply by using the definition of $f_{j}$ for $j\in
I_{0,0, \dots, r_{k}}$ with $r_{k}=1$:

\begin{fact}\label{fact4}
Let $\delta _{n}$ be a small positive number. If $\gamma _{n}$ is small
enough, then for every $x$ supported in $[0,\xi _{n}]$ and every $1\leq k
\leq k_{n}$,
$$||T^{c_{k,n}}x-p_{k,n}(T)x||\leq \delta _{n}||x||.$$
\end{fact}

\begin{proof}
Since $\textrm{sp}[e_{0},\ldots, e_{\xi _{n}}]=
\textrm{sp}[f_{0},\ldots, f_{\xi _{n}}]$,
we can write any vector $x$ with support in $[0,\xi _{n}]$ as
$x=\sum_{j=0}^{\xi _{n}}\alpha
_{j}^{(n)}e_{j}$. Then $T^{c_{k,n}}x=\sum_{j=0}^{\xi _{n}}\alpha
_{j}^{(n)}e_{j+c_{k,n}}$. Now $j+c_{k,n}$
belongs to the working interval $[c_{k,n},c_{k,n}+\xi
_{n}]$, so $f_{j+c_{k,n}}=
{\gamma _{n}}^{-1}(e_{j+c_{k,n}}-p_{k,n}(T)e_{j})$. Hence
$$T^{c_{k,n}}x=\gamma _{n}\sum_{j=0}^{\xi _{n}}\alpha _{j}^{(n)}
f_{j+c_{k,n}}+p_{k,n}(T)\sum_{j=0}^{\xi _{n}}
\alpha _{j}^{(n)}e_{j},$$ that is $$||T^{c_{k,n}}x-p_{k,n}
(T)x||=||\gamma _{n}\sum_{j=0}^{\xi _{n}}\alpha
_{j}^{(n)}f_{j+c_{k,n}}||\leq \gamma _{n}
(\sum_{j=0}^{\xi _{n}}|\alpha _{j}^{(n)}|^{2})^{\frac {1}{2}}.$$ On the
space $F_{\xi _{n}}=\textrm{sp}[f_{0},\ldots, f_{\xi
_{n}}]$, the two norms $||x||_{0}=(\sum_{j=0}^{\xi _{n}}|
\alpha _{j}^{(n)}|^{2})^{\frac {1}{2}}$
 and $||x||$ are equivalent, so there
 exists a constant $C_{\xi _{n}}$
depending only on $\xi _{n}$ such that $||x||_{0}\leq C_{\xi _{n}}||x||$
for every $x$ supported in $[0,\xi _{n}]$. Thus
$||T^{c_{k,n}}x-p_{k,n}(T)x||\leq \gamma _{n}C_{\xi _{n}}||x||
\leq \delta _{n}||x||$ if $\gamma _{n}$ is small enough.
\end{proof}

Hence if the collection $(p_{k,n})$ is ``sufficiently dense'' among
polynomials with $|p|\leq 2$, Fact \ref{fact4} gives that the orbit of
the vector $x=e_{0}$ (and hence of any finitely supported vector $x$)
contains in its closure any vector $p(T)x$ with $|p|\leq 2$.
 In order to obtain this result for every vector, not only
finitely supported ones, one clearly needs to control the behaviour of
the quantities $||T^{c_{k,n}}(x-\pi_{[0,\xi_{n}]}x)||$ (and then to
dispense with the condition $|p|\leq 2$, but this is not difficult). More
precisely, we will need the following proposition, which we shall prove
in Section $3$:

\begin{proposition}\label{prop2}\label{prop5}
For every $n\geq 1$, every $1\leq k \leq k_{n}$ and every $x\in H$ such
that $\pi_{[0,\xi _{n}]}x=0$, $||T^{c_{k,n}}x||\leq 100\,||x||$. In other
words, $$||T^{c_{k,n}}(x-\pi_{[0,\xi _{n}]}x)||\leq \,100 ||x-\pi_{[0,\xi
_{n}]}x|| $$ for every $x\in H$.
\end{proposition}

Only the intervals $I_{0,\dots,0, 1}$ are needed for the proof of Fact
\ref{fact4}, but for the estimates of Proposition \ref{prop5} one needs
the whole lattice, and this is why all the other intervals, which could
be called ``shades'' of the basic intervals $I_{0,\dots,0, 1}$, appear in
the definition of the fan.
\par\smallskip
We finish this section by showing how $e_{j}$ can be computed for $j$ in
a working interval by going down the lattice along each successive
coordinate:

\begin{lemma}\label{lem7}
For every $\alpha \in [0,\xi _{n}]$ and every $k_{n}$-tuple
$(r_{1},\dots, r_{k_{n}})$ of integers in $[0,h_{n}]$,
\begin{eqnarray*}
e_{r_{1}c_{1,n}+\dots+r_{t}c_{t,n}+\alpha }&=&\Bigl(\sum_{l=1}^{t}
\sum_{s_{l}=0}^{r_{l}}
\gamma _{n} 4^{r_{1}+\dots r_{l-1}+(r_{l}-s_{l})-1}
p_{l,n}(T)^{s_{l}}p_{l+1,n}(T)^{r_{l+1}}\dots p_{t,n}(T)^{r_{t}} \\ &&
f_{r_{1}c_{1,n}+\dots+
r_{l-1}c_{l-1,n}+(r_{l}-s_{l})c_{l,n}+\alpha}\Bigr) +
p_{1,n}(T)^{r_{1}}\dots p_{t,n}(T)^{r_{t}}e_{\alpha },
\end{eqnarray*}
 where $t$ is the
largest index such that $r_{t}\geq 1$.
\end{lemma}

\begin{proof}
We have
\begin{eqnarray*}
e_{r_{1}c_{1,n}+\dots+r_{t}c_{t,n}+\alpha }&=&\gamma _{n}4^{|r|-1}
f_{r_{1}c_{1,n}+\dots+ r_{t}c_{t,n}+\alpha
}+p_{t,n}(T)e_{r_{1}c_{1,n}+\dots+ (r_{t}-1)c_{t,n}+\alpha }=\dots\\
&=&\gamma _{n}4^{|r|-1}\sum_{s_{t}=0}^{r_{t}}4^{-s_{t}}
p_{t,n}(T)^{s_{t}} f_{r_{1}c_{1,n}+\dots+ (r_{t}-s_{t})c_{t,n}+\alpha }\\
&+&
 p_{t,n}(T)^{r_{t}}e_{r_{1}c_{1,n}+\dots+
+r_{t-1}c_{t-1,n}+\alpha }.
\end{eqnarray*}
Then we go down in the same way along the $(t-1)$-coordinate, etc... until
there are no more coordinates left.
\end{proof}

We will always choose the maximal degree $d_{n}$ of the polynomials
$p_{k,n}$ to be small \wrt\ $c_{1,n}$: for the proof of Theorem \ref{th1}
we will choose simply $d_{n}=n$.

\subsection{Construction of $f_{j}$ for $j$ in a \loi}
The \loi s are the intervals which lie between the working intervals. If
we write such an interval as $[r+1, r+s]$ $f_{j}$ is defined for $j$ in
it
 as $f_{j}=\lambda _{j}e_{j}$, where
$$\lambda _{j}=2^{(\frac {1}{2}s+r+1-j)/\sqrt{s}}.$$
When the length $s$ of such a \loi\ becomes very large, the coefficients
$\lambda _{j}$ behave in the following way: if  $j$ lies in the beginning
of the interval, $\lambda_{j}$ is roughly equal to $2^{\frac
{1}{2}\sqrt{s}}$ (very large), and when $j$ is near the end of the
\loi, $\lambda _{j}$ is roughly $2^{-\frac {1}{2}\sqrt{s}}$ (very
small). This implies in particular that when $j$ is in the beginning of a
\loi, $||e_{j}||$ is very small, approximately less than $2^{-\frac
{1}{2}\sqrt{s}}$.
 Moreover if $s$ is large, the ratio $\lambda _{j}/\lambda _{j+1}$
for $j$ and $j+1$ is the \loi\ becomes very close to $1$. Remark that
this ratio does not depend on $j$.
\par\smallskip
Hence the picture at step $n$ is the following: there is first one very
large \loi, between $\xi _{n}+1$ and $c_{1,n}-1$, then an alternance of
working and \loi s, and at the end a very large \loi\ between $h_{n}
(c_{1,n}+\dots, c_{k_{n},n})+\xi _{n}+1$ and $\xi _{n+1}$.
Then the length of all the \loi s between working
intervals is always comparable to some $c_{k,n}$, the length of the first
\loi\ $[\xi _{n}+1, c_{1,n}-1]$ is comparable to $c_{1,n}$, and the
length of the last one is comparable to $\xi _{n+1}$. Since it would make
the computations too involved if we were to write each time the precise
estimates for $\lambda _{j}$ or $||e_{j}||$, we will often  write only an
approximate  estimate which will give the order of magnitude of the
quantities involved. When doing this, we will use the symbol $\ls$
instead of $\leq$, or $\gs$ instead of $\geq$. For instance for $j$ in the
 beginning of the \loi\
$
[\xi _{n}+1,c_{1,n}-1]$, let us say $j\in [\xi_{n}+1,2\xi _{n}+1]$,
 we will not write
$$||e_{j}||\leq 2^{-\frac {1}{\sqrt{c_{1,n}-\xi _{n}-1}}(\frac {1}{2}
(c_{1,n}-\xi _{n}-1)+\xi _{n}+1-j)}
\leq 2^{-\frac {1}{\sqrt{c_{1,n}-\xi _{n}-1}}(\frac {1}{2}
(c_{1,n}-\xi _{n}-1)-\xi _{n})},$$ but simply
 $||e_{j}||
\ls 2^{-\frac {1}{2}\sqrt{c_{1,n}}}$, and since
$\xi _{n}$ and $d_{n}$ are both small \wrt\ $c_{1,n}$, the estimate
$2^{-\frac {1}{2}\sqrt{c_{1,n}}}$ gives the right order of magnitude for
$||e_{j}||$.

\subsection{Boundedness of the operator $T$}
In order to show that $T$ is bounded on $H$, we need the following
estimates:

\begin{proposition}\label{prop0}\label{prop6}
Let $(\delta _{n})_{n\geq 0}$ be a decreasing sequence of positive
numbers going to zero very fast. The  vectors $f_{j} $ can be constructed
so that for every $n\geq 0$, assertion (1) below holds true:

\begin{enumerate}
\item[(1)] if $x$ is supported in the interval $[\xi _{n}+1, \xi
_{n+1}]$, then

\begin{enumerate}
\item[(1a)] $||\pi_{[\xi _{n}+1, \xi
_{n+1}]}(Tx)||\leq (1+\delta _{n}) ||x||$

\item[(1b)] $||\pi_{[0, \xi
_{n}]}(Tx)||\leq \delta _{n} ||x||$.
\end{enumerate}
\end{enumerate}
\end{proposition}

Remark that since $[\frac {1}{2}\xi _{n+1},2\xi _{n+1}]$ can be supposed
to be contained in a \loi, it makes sense to write $\pi_{ [\xi _{n}+1,\xi
_{n+1}]}(Tx)$, even when $x$ has a non-zero coordinate on $\xi _{n+1}$.
If $x=f_{\xi _{n+1}} $ for instance, we know, even if $\lambda _{\xi
_{n+1}+1}$ has not been defined yet, that $Tf_{\xi _{n+1}}$ is a multiple
of $f_{\xi _{n+1}+1}$, and thus the projection of $Tx$ on $[\xi
_{n}+1,\xi _{n+1}]$ is zero.

\begin{proof}
Write the vector $x$ as $x=\sum_{j=\xi _{n}+1}^{\xi _{n+1}}x_{j}f_{j}$,
and its image as
 $Tx=
\sum_{j=\xi _{n}+1}^{\xi _{n+1}}x_{j}Tf_{j}$. There are four kind of
indices $j$ in this sum, with a different expression for $Tf_{j}$ each
time.
\par\smallskip
$\bullet$ Let $J_{1}$ be the set of integers $j\in[\xi _{n}+1,\xi
_{n+1}]$ such that $j$ and $j+1$ belong to a \loi: $f_{j}=\lambda
_{j}e_{j}$ and $f_{j+1}=
\lambda _{j+1}e_{j+1}$, so that $Tf_{j}=(\lambda _{j}/\lambda
_{j+1})f_{j+1}$. If the length of the \loi\ is very large, $\lambda _{j}/\lambda
_{j+1}\leq 1+\delta _{n}/2$ for every $j\in J_{1}$, and $Tf_{j}=\mu_{j}f_{j+1}$
with $|\mu _{j}|\leq 1+\delta _{n}/2$.
\par\smallskip
$\bullet$ Let $J_{2}$ be the set of integers $j\in[\xi _{n}+1,\xi
_{n+1}]$ such that $j$ and $j+1$ belong to a working interval: then
simply $Tf_{j}=f_{j+1}$.
\par\smallskip
$\bullet$ Let $J_{3}$ be the set of integers $j\in[\xi _{n}+1,\xi
_{n+1}]$ of the form $j=r_{1}c_{1,n}+\dots+r_{t}c_{t,n}+\xi _{n}$: $j$ is
the endpoint of a working interval and $j+1$ is the first point of the
next \loi. Then $$Tf_{j}=
\gamma _{n}^{-1}4^{1-|r|}\left(
e_{j+1}-p_{t,n}(T)e_{j-c_{t,n}+1}\right).$$ We have $||e_{j+1}||\ls
2^{-\frac {1}{2}\sqrt{c_{1,n}}}$. Moreover if we write the polynomial
$p_{t,n}$ as $p_{t,n}(\zeta )=\sum_{u=0}^{d_{n}}a_{u}\zeta ^{u}$, then $
p_{t,n}(T)e_{j-c_{t,n}+1}=\sum_{u=0}^{d_{n}}a_{u}e_{j-c_{t,n}+1+u}$. Now
since $d_{n}$ is very small \wrt\ each
 $c _{k,n}$, $j-c_{t,n}+1+u$
lies in the beginning of a \loi, and thus $||e_{j-c_{t,n}+1+u}||\ls
2^{-\frac {1}{2}\sqrt{c_{1,n}}}$, so $$||p_{t,n}(T)e_{j-c_{t,n}+1}||\leq
2 \sup_{0\leq u\leq d_{n}} ||e_{j-c_{t,n}+1}||\ls 2^{-\frac
{1}{2}\sqrt{c_{1,n}}}.$$
 Hence $||Tf_{j}||\ls \gamma _{n}^{-1}
2^{-\frac {1}{2}\sqrt{c_{1,n}}}$ and since $\gamma _{n}$ depends only on
$\xi
_{n}$, $||Tf_{j}||$ can be made arbitrarily small for an appropriate
choice of $c_{1,n}$.
\par\smallskip
$\bullet$ Let $J_{4}$ be the set of integers $j\in[\xi _{n}+1,\xi
_{n+1}]$ of the form $j=r_{1}c_{1,n}+\dots+r_{t}c_{t,n}-1$: $j$ is the
endpoint of a \loi, and $j+1$ is the first endpoint of the next working
interval. Then $Tf_{j}=\lambda _{j}e_{j+1}$. Using Lemma \ref{lem7}, we
get that
\begin{eqnarray*}
e_{r_{1}c_{1,n}+\dots+r_{t}c_{t,n}}&=&\Bigl(\sum_{l=1}^{t}
\sum_{s_{l}=0}^{r_{l}}
\gamma _{n} 4^{r_{1}+\dots r_{l-1}+(r_{l}-s_{l})-1}
p_{l,n}(T)^{s_{l}}p_{l+1,n}(T)^{r_{l+1}}\dots p_{t,n}(T)^{r_{t}} \\ &&
f_{r_{1}c_{1,n}+\dots+ r_{l-1}c_{l-1,n}+(r_{l}-s_{l})c_{l,n}}\Bigr) +
p_{1,n}(T)^{r_{1}}\dots p_{t,n}(T)^{r_{t}}e_{0 }.
\end{eqnarray*}
The polynomial $p_{s_{1},\dots, s_{t}}=p_{1,n}^{s_{1}}\dots
p_{t,n}^{s_{t}}$ has degree at most $h_{n}k_{n}d_{n} $, and $|p|\leq
2^{s_{1}+\dots+s_{t}}$. Write $p_{s_{1},\dots, s_{t}}(\zeta
)=\sum_{u=0}^{h_{n}k_{n}d_{n}} a_{u}^{(s_{1},\dots,s_{t})}\zeta ^{u}$.
Then

\begin{eqnarray*}
&&p_{0,\dots,0, s_{l}, r_{l+1},\dots, r_{t}}(T) f_{r_{1}c_{1,n}+\dots+
r_{l-1}c_{l-1,n}+(r_{l}-s_{l})c_{l,n}}=\qquad\\
&&\qquad\qquad\qquad\sum_{u=0}^{h_{n}k_{n}d_{n}} a_{u}^{(0,\dots,0,
s_{l}, r_{l+1}, \dots, r_{t})}T^{u}f_{r_{1}c_{1,n}+\dots+
r_{l-1}c_{l-1,n}+(r_{l}-s_{l})c_{l,n}}.
\end{eqnarray*}

If $u\leq \xi _{n}$, $$T^{u}f_{r_{1}c_{1,n}+\dots+
r_{l-1}c_{l-1,n}+(r_{l}-s_{l})c_{l,n}}=f_{r_{1}c_{1,n}+\dots+
r_{l-1}c_{l-1,n}+(r_{l}-s_{l})c_{l,n}+u}$$ and if $\xi _{n}+1\leq
 u\leq h_{n}k_{n}d_{n}$, then
 $$T^{u}f_{r_{1}c_{1,n}+\dots+
r_{l-1}c_{l-1,n}+(r_{l}-s_{l})c_{l,n}}=T^{\alpha } f_{r_{1}c_{1,n}+\dots+
r_{l-1}c_{l-1,n}+(r_{l}-s_{l})c_{l,n}+\xi _{n}}$$ where $1\leq \alpha
\leq h_{n}k_{n} d_{n}$. So if $r'=r_{1}c_{1,n}+\dots+
r_{l-1}c_{l-1,n}+(r_{l}-s_{l})c_{l,n}$, then in the case where $r'\not
=0$,
\begin{eqnarray*}
 T^{u}f_{r_{1}c_{1,n}+\dots+
r_{l-1}c_{l-1,n}+(r_{l}-s_{l})c_{l,n}}&=&\gamma _{n}^{-1}
4^{1-|r'|}e_{r_{1}c_{1,n}+\dots+
r_{l-1}c_{l-1,n}+(r_{l}-s_{l})c_{l,n}+\alpha }\\ &-&
p_{v,n}(T)e_{r_{1}c_{1,n}+\dots+
r_{l-1}c_{l-1,n}+(r_{l}-s_{l})c_{l,n}-c_{v,n}+\alpha }
\end{eqnarray*}
where $v$ is the largest non-zero coordinate in $r'$. Using exactly the
same argument as in the case $j\in J_{3}$ above, we see that in the case
where $\xi _{n}+1\leq
 u\leq h_{n}k_{n}d_{n}$, then $||T^{u}f_{r_{1}c_{1,n}+\dots+
r_{l-1}c_{l-1,n}+(r_{l}-s_{l})c_{l,n}}||$ can be made arbitrarily small.
When $r'=0$,
 $||T^{u}e_{0}||=||e_{u}||$, and with $\xi _{n}+1
\leq u \leq h_{n}k_{n}d_{n}$, $||e_{u}||$ can be made
arbitrarily small again. Hence
\begin{eqnarray*}
&&\left|\left|\sum_{l=1}^{t}
\sum_{s_{l}=0}^{r_{l}}
\gamma _{n} 4^{r_{1}+\dots r_{l-1}+(r_{l}-s_{l})-1}
p_{0,\dots,0, s_{l}, r_{l+1},\dots, r_{t}}(T) f_{r_{1}c_{1,n}+\dots+
r_{l-1}c_{l-1,n}+(r_{l}-s_{l})c_{l,n}}\right|\right|\\ &&\qquad\qquad\ls
\sum_{l=1}^{t}
\sum_{s_{l}=0}^{r_{l}}
\gamma _{n} 4^{r_{1}+\dots r_{l-1}+(r_{l}-s_{l})-1}2^{s_{l}+r_{l+1}
+\ldots r_{t}}\leq \gamma _{n}4^{|r|-1}2k_{n}.
\end{eqnarray*}
For the remaining term $||p_{r_{1},\dots, r_{t}}(T)e_{0}|| $ we
proceed as above:
\begin{eqnarray*}
||p_{r_{1},\dots, r_{t}}(T)e_{0}||&\leq&
||\sum_{u=0}^{h_{n}k_{n}d_{n}} a_{u}^{(r_{1},\dots,r_{t})}e_{u}||\\
&\leq& \left(\sum_{u=0}^{h_{n}k_{n}d_{n}} |a_{u}^{(r_{1},\dots,r_{t})}|
\right)\; \sup_{
u\leq h_{n}k_{n}d_{n}}||e_{u}||\leq 2^{h_{n}k_{n}} \sup_{
u\leq h_{n}k_{n}d_{n}}||e_{u}||.
\end{eqnarray*}
Since $\lambda _{j}
\ls 2^{-\frac {1}{2}\sqrt{c_{1,n}}}$
and neither $h_{n}$ nor $k_{n}$ nor $d_{n}$ depend on $c_{1,n}$,
 we obtain that $||p_{r_{1},\dots, r_{t}}(T)e_{0}||$
can be made arbitrarily small, and hence the same is true for
$||Tf_{j}||$.
\par\smallskip
$\bullet$ Putting the previous estimates together, we obtain that
$$||   T(
\sum_{j\in J_{1}\cup J_{2}}x_{j}f_{j}
) ||^{2}\leq (1+\frac {\delta _{n}}{2})^{2}
\sum_{j\in J_{1}}|x_{j}|^{2}+ \sum_{j\in J_{2}}|x_{j}|^{2}
\leq (1+\frac {\delta _{n}}{2})^{2} ||x||^{2}
$$
and
$$||T(\sum_{j\in J_{3}\cup J_{4}}x_{j}f_{j}
)||\leq
\left(\sum_{j\in J_{3}\cup J_{4}}||Tf_{j}||^{2}\right)
^{\frac {1}{2}}
\left(\sum_{j\in J_{3}\cup J_{4}}|x_{j}|^{2}\right)^{\frac {1}{2}}
\leq \frac {\delta _{n}}{2}||x||,
$$
so that $||Tx||\leq (1+\delta _{n})||x||$, and this proves that
$||\pi_{[\xi _{n}+1,
\xi_{n+1} ]}(Tx)||\leq (1+\delta _{n})||x||$ which proves (1a). Since $
\pi_{[0,
\xi_{n} ]}\left(\sum_{j\in J_{1}\cup J_{2}}x_{j}f_{j}\right)=0$, this
proves (1b) too.
\end{proof}

The boundedness of $T$ follows now easily from Proposition \ref{prop6}:

\begin{proposition}\label{prop1}\label{prop8}
Let $\varepsilon $ be any positive number. If the sequence $(\delta
_{n})$ corresponding to the construction of Proposition
\ref{prop6} goes fast enough to zero, $T$ extends to a bounded \op\ on $H$
satisfying $||T||\leq 1+\varepsilon $.
\end{proposition}

\begin{proof}
The proof is by induction on $n$, supposing that $||Tx||\leq C_{n}||x||$
for every $x$ supported in $[0,\xi _{n}]$. Suppose that $x$ is supported
in $[0,\xi _{n+1}]$, and write $Tx$ (which is supported in $[0,\xi
_{n+2}]$) as
\begin{eqnarray*}
Tx&=&T(\pi_{[0,\xi _{n}]}(x))+T(\pi_{[\xi _{n}+1,\xi _{n+1}]}(x))
=
\pi_{[0,\xi _{n}]} (T(\pi_{[0,\xi _{n}]}(x)))\\
&+&\pi_{[\xi _{n}+1,\xi _{n+1}]} (T(\pi_{[0,\xi _{n}]}(x)))+\pi_{[0,\xi
_{n}]}(T(
\pi_{[\xi _{n}+1,\xi _{n+1}]}(x)))\\
&+&\pi_{[\xi _{n}+1,\xi _{n+1}]} (T(\pi_{[\xi _{n}+1,\xi _{n+1}]}(x))) +
\pi_{[\xi _{n+1}+1,\xi _{n+2}]} (T(\pi_{[\xi _{n}+1,\xi _{n+1}]}(x))).
\end{eqnarray*}
Hence
\begin{eqnarray*}
||Tx||^{2}&=& ||\pi_{[0,\xi _{n}]} (T(\pi_{[0,\xi _{n}]}(x)))+\pi_{[0,\xi
_{n}]}(T(
\pi_{[\xi _{n}+1,\xi _{n+1}]}(x)))||^{2}\\
&+&||\pi_{[\xi _{n}+1,\xi _{n+1}]} (T(\pi_{[0,\xi _{n}]}(x))) +\pi_{[\xi
_{n}+1,\xi _{n+1}]} (T(\pi_{[\xi _{n}+1,\xi _{n+1}]}(x)))||^{2}\\ &+&
||\pi_{[\xi _{n+1}+1,\xi _{n+2}]} (T(\pi_{[\xi _{n}+1,\xi
_{n+1}]}(x)))||^{2}.
\end{eqnarray*}
The terms in this expression which remain to be estimated
are $||\pi_{[\xi _{n}+1,\xi _{n+1}]} (T(\pi_{[0,\xi _{n}]}(x)))||$
and $||\pi_{[\xi _{n+1}+1,\xi _{n+2}]} (T(\pi_{[\xi _{n}+1,\xi
_{n+1}]}(x)))||$. The first one is equal to
$|x_{\xi _{n}+1}|^{2}\,||Tf_{\xi _{n+1}}||^{2}=|x_{\xi _{n}+1}|^{2}
(\lambda _{\xi _{n+1}}/\lambda _{\xi _{n+1}+1})^{2}$, and we can
choose $\lambda _{\xi _{n+1}+1}$ so large that
$\lambda _{\xi _{n+1}}/\lambda _{\xi _{n+1}+1}\leq \delta _{n-1}$ for
instance. We do the same for the last term, and then
\begin{eqnarray*}
||Tx||^{2}&\leq& (C_{n}||\pi_{[0,\xi _{n}]}(x)||+
\delta _{n}||\pi_{[\xi_{n}+1,\xi _{n+1}]}(x)||)^{2}\\
&+& (\delta _{n-1}||\pi_{[0,\xi _{n}]}(x)||+(1+\delta _{n})||
\pi_{[\xi_{n}+1,\xi _{n+1}]}(x)||)^{2}+\delta _{n}^{2}||
\pi_{[\xi_{n}+1,\xi _{n+1}]}(x)||^{2}\\
&=&(C_{n}^{2}+\delta _{n-1}^{2}) ||\pi_{[0,\xi _{n}]}(x)||^{2}+ (1+(1+\delta _{n})^{2}
+\delta _{n}^{2}) ||\pi_{[\xi_{n}+1,\xi
_{n+1}]}(x)||^{2}\\
&+& (2C_{n}\delta _{n}+2\delta _{n-1}(1+\delta _{n})) ||\pi_{[0,\xi _{n}]}(x)||\,
||\pi_{[\xi_{n}+1,\xi _{n+1}]}(x)||
\end{eqnarray*}
which  yields that
$$||Tx||^{2}
\leq \left(
(\max (C_{n}^{2}+\delta _{n-1}^{2}, 1+(1+\delta _{n})^{2}+\delta _{n}^{2}) +
(2C_{n}\delta _{n}+2\delta _{n-1}(1+\delta _{n}))\delta _{n}\right) ||x||^{2},$$ and the proof of
Proposition \ref{prop1} follows by induction.
\end{proof}

We finish this section with the following stronger form of
Proposition \ref{prop6}:

\begin{proposition}\label{prop00}\label{prop9}
 Given a sequence of positive numbers $(\varepsilon_n
)_{n\geq 1}$ which decreases very quickly to zero, the construction of
the fans at each step can be conducted in such a way that

\begin{enumerate}
\item[(1')] if $x$ is supported in the interval $[\xi _{n}+1, \xi
_{n+1}]$, then for every $m<\xi _{n}/2$

\begin{enumerate}
\item[(1a')] $||\pi_{[\xi _{n}+1, \xi
_{n+1}]}(T^{m}x)||\leq (1+\varepsilon  _{n}) ||x||$

\item[(1b')] $||\pi_{[0, \xi
_{n}]}(T^{m}x)||\leq \varepsilon  _{n} ||x||$

\item[(1c')] $\|\pi_{[\xi_{n+1}+1,\xi_{n+2}]}(T^mx)\|\leq (1+\varepsilon_n)
\|x \|$.
\end{enumerate}
\end{enumerate}
\end{proposition}

\begin{proof}
As in the proof of  Proposition
\ref{prop6}, we are going to show that
if the construction has been carried out until step $n-1$, the
$c_{j,n}$'s at step $n$ can be chosen so large
 that (1a') and (1b') hold true at step $n$, as well as
(1c') at step $n-1$.
We denote again $F_{\xi_n}=\textrm{sp}[e_0,\dots,e_{\xi_n}]$.
As soon as $c_{1,n}$ is much larger than $\xi
_{n}$, the projection on $[\xi _{n}+1,\xi_{n+1}]$ of $T^{m}(F_{\xi
_{n}})$, $m<\xi _{n}/2$, consists of vectors supported in the beginning
of the \loi\ $[\xi _{n}+1,c_{1,n}-1]$. This implies that
$\|\pi_{[\xi_{n+1}+1,\xi_{n+2}]}(T^mx)\|
\leq C_{\xi _{n}}2^{-\frac {1}{2}\sqrt{c_{1,n}}}\|x\|$
where $C_{\xi_{n}}$ depends only on the steps $0$ to $n-1$
while $c_{1,n}$ is very large \wrt\ $C_{\xi _{n}}$:
this shows that condition (1c') at step $n-1$ is satisfied.
\par\smallskip
Denote by $T_{\xi _{n}}$ the truncated shift on $F_{\xi _{n}}$ \wrt\ the
vectors $e_{j}$: $T_{\xi _{n}}e_{j}=e_{j+1}$ for $j<\xi _{n}$ and
$T_{\xi _{n}}e_{\xi _{n}}=0$. The
proof of Proposition 2.5 shows that one can ensure
 that $\|T_{\xi _{n}}\|\leq
2-\frac1{n}$ for instance. The fact that
conditions (1a') and (1b') can be fulfilled follows from the
statement $(P_m)$ below which we prove by induction:
\par\smallskip\par\smallskip
$(P_m)$: there  exists a constant $C_{m,n}$ depending only on
the construction until step $n-1$ such that if
properties (1a) and (1b) of Proposition \ref{prop6} at step $n$
are satisfied for for some $\delta_n>0$, then for every $x$ supported in
$[\xi _{n}+1,\xi _{n+1}]$,
$\|\pi_{[\xi_{n}+1,\xi_{n+1}]}(T^mx)\|\leq (1+C_{n,m}\delta_n) \|x
\|$ and $\|\pi_{[0,\xi_{n}]}(T^mx)\|\leq C_{m,n}\delta_n\|x \|$.
\par\smallskip\par\smallskip
Once $(P_m)$ is proven, it suffices to choose $\delta_n=\varepsilon_{n}/
\max_{m<\xi _{n}/2}(C_{n,m})$.
The base of the inductive proof of $(P_n)$ is Proposition \ref{prop6}
itself. Assume now that $(P_{m-1})$ holds true. Write $\pi_{[\xi
_{n}+1, \xi _{n+1}]}(T^{m}x)=\pi_{[\xi _{n}+1, \xi
_{n+1}]}(T(T^{m-1}x)).$ If $y=T^{m-1}x$, $y$ is supported in $[0,\xi
_{n+1} -m+1]$ and we have
\begin{eqnarray*}
\pi_{[\xi _{n}+1, \xi _{n+1}]}(Ty)&=&\pi_{[\xi _{n}+1, \xi
_{n+1}]}(T(\pi_{[0,\xi _{n}]}(y)))+\pi_{[\xi _{n}+1, \xi
_{n+1}]}(T(\pi_{[\xi_{n}+1,\xi
_{n+1}]}(y)))\\
&+&\pi_{[\xi _{n}+1, \xi _{n+1}]}(T(\pi_{[\xi _{n+1}+1,\xi
_{n+1}+m-1]}(y)))
\end{eqnarray*}
and
\begin{eqnarray*}
\pi_{[0, \xi _{n}]}(Ty)&=&\pi_{[0, \xi _{n}]}(T(\pi_{[0,\xi
_{n}]}(y)))+\pi_{[0, \xi
_{n}]}(T(\pi_{[\xi_{n}+1,\xi _{n+1}]}(y)))\\
&+&\pi_{[0, \xi _{n}]}(T(\pi_{[\xi _{n+1}+1,\xi _{n+1}+m-1]}(y))).
\end{eqnarray*}
Since the vector $\pi_{[\xi _{n+1}+1,\xi _{n+1}+m-1]}(y)$ is supported on the
first lay-off interval of $[\xi _{n+1}+1,\xi _{n+2}]$, the operator
$T$ acts on it as a weighted shift operator and the projection
$$\pi_{[0, \xi _{n+1}]}(T(\pi_{[\xi _{n+1}+1,\xi _{n+1}+m-1]}(y))),$$
as well as the last term in each one of the two displays above is zero.
For the other two terms we have (assuming that $\delta_n<1$) $$\|\pi_{[\xi
_{n}+1, \xi _{n+1}]}(Ty)\|\leq \|T(\pi_{[0,\xi
_{n}]}(y))\|+(1+\delta_n)\|\pi_{[\xi_{n}+1,\xi _{n+1}]}(y)\|\leq $$
$$ 2C_{m-1,n}\delta_n\|x\|+ (1+\delta_n) \|\pi_{[\xi_{n}+1,\xi
_{n+1}]}(y)\|\leq (1+(3C_{m-1,n}+2)\delta_n)\|x\|$$ and
$$\|\pi_{[0, \xi _{n}]}(Ty) \|\leq \|T(Y)\|+
\delta_n \|\pi_{[\xi_{n}+1,\xi _{n+1}]}(y)\|\leq $$
$$2\|\pi_{[0,\xi _{n}]}(y)\|+\delta_n(1+C_{m-1,n}\delta_n)\|x\|\leq
(3C_{m-1,n}+1)\delta_n\|x\|,$$ which completes the induction and
thus the proof of Proposition \ref{prop9}.
\end{proof}

\section{Estimating $T^{c_{k,n}}$: proof of Theorem \ref{th1}}
As was already mentioned before, the crucial step for the proof of
Theorem \ref{th1} is Proposition \ref{prop5}. The
estimates needed for this are given in Proposition \ref{prop10}:

\begin{proposition}\label{prop10}
Let $(\delta _{n})_{n\geq 0}$ be a decreasing sequence of positive
numbers going to zero very fast. The  vectors $f_{j} $ can be constructed
so that for every $n\geq 0$, assertion (2) below holds true:

\begin{enumerate}
\item[(2)] for any vector $x$  supported in the interval $[\xi _{n}+1, \xi
_{n+1}]$ and for any $1\leq k \leq k_{n}$,
\begin{enumerate}
\item[(2a)] $||\pi_{[\xi _{n}+1, \xi
_{n+1}]}(T^{c_{k,n}}x)||\leq 4 ||x||$

\item[(2b)] $||\pi_{[0, \xi
_{n}]}(T^{c_{k,n}}x)||\leq \delta _{n} ||x||$
\end{enumerate}

\item[(3)] for any $x$ supported in the interval $[0, \xi
_{n}]$ and any $m<\xi _{n}/2$,
\begin{enumerate}
\item[$\;$]$||\pi_{[\xi _{n}+1, \xi
_{n+1}]}(T^{m}x)||\leq \delta _{n} ||x||$.
\end{enumerate}
\end{enumerate}
\end{proposition}

\begin{proof}
$\bullet$ The easy part of the proof is assertion (3): if $x$ is
supported in $[0,\xi _{n}]$ and $m<\xi _{n}/2$, $T^{m}x$ is supported in
the interval $[0, (3/2)\xi
_{n}]$. If $x=\sum_{j=0}^{\xi _{n}}\alpha
_{j}e_{j}$, $T^{m}x=\sum_{j=0}^{\xi _{n}}\alpha
_{j}e_{j+m}$. Now $e_{j+m}$ can have a non-zero component in the
interval $[\xi _{n}+1, \xi _{n+1}]$ only in the case where $j+m\in [\xi
_{n}+1, \xi
_{n}+m]\subseteq[\xi _{n}+1, (3/2)\xi_{n}]$, so
$
\pi_{[\xi _{n}+1,\xi _{n+1}]}(T^{m}x)=\pi_{[\xi _{n}+1,\xi _{n+1}]}
(\sum_{j=\xi _{n}+1-m}^{\xi _{n}}\alpha _{j}e_{j+m}).$ But $[\xi _{n}+1,
(3/2)\xi_{n}]$ is contained in beginning of the first \loi\ $[\xi _{n}+1,
c_{1,n}]$ if $c_{1,n}$ is large enough, so $||e_{j+m}||$ is very small,
and $||\pi_{[\xi _{n}+1,\xi _{n+1}]}(T^{m}x)||\ls 2^{-\frac
{1}{2}\sqrt{c_{1,n}}}$ which can be made smaller than $\delta_{n}$.
\par\smallskip
$\bullet$ Fix $1\leq k\leq k_{n}$. Let us first look at
$T^{c_{k,n}}f_{j}$ for $j$ in a \loi. Since it would be rather intricate
to write down all the possible cases, we give an example of each one of
the situations which can occur:
\par\smallskip
-- if $j\in c_{k,n}+\xi _{n}+1, 2c_{k,n}-1]$, then $T^{c_{k,n}}f_{j}=
\lambda _{j}e_{j+c_{k,n}}$ and $j+c_{k,n}\in [2c_{k,n}+\xi _{n}+1, 3c_{k,n}
-1]$. Thus $e_{j+c_{k,n}}=(1/{\lambda _{j+c_{k,n}}})f_{j+c_{k,n}}$
and since $\lambda _{j}=\lambda _{j+c_{k,n}}$, $T^{c_{k,n}}f_{j}=
f_{j+c_{k,n}}$.
\par\smallskip
-- if $j\in [h_{n}c_{k,n}+\xi _{n}+1, c_{1,n}+h_{n}c_{k,n}-1]$, $T^{c_{k,n}}f_{j}=
({\lambda _{j}}/{\lambda _{j+c_{k,n}}})f_{j+c_{k,n}}$ and $j+c_{k,n}\in
[(h_{n}+1)c_{k,n}+\xi _{n}+1, c_{1,n}+(h_{n}+1)c_{k,n}-1]$ which is
contained in the beginning of the \loi\
$[h_{n}(c_{1,n}+\dots+c_{k,n})+\xi _{n}+1, c_{k+1,n}-1]$ whose length is
approximately less than $2^{\sqrt{c_{k+1,n}}}$. Hence $|\lambda _{j}|\ls
2^{\frac {1}{2}
\sqrt{c_{1,n}}}$
and $|\lambda _{j+c_{k,n}}|\gs 2^{\sqrt{c_{k+1,n}}}$ so the quotient
$\lambda _{j}/\lambda _{j+c_{k,n}}$ is extremely small.
\par\smallskip
All the situations which can occur reproduce one of these two situations,
and we leave the reader to work out the details by himself.
\par\smallskip
$\bullet$ Then let us consider the case where $j$ belongs to a working
interval $I_{r_{1},\dots, r_{t}}$ with $k\leq t$: $T^{c_{k,n}}f_{j}=
\gamma _{n}^{-1}4^{1-|r|}
(e_{j+c_{k,n}}-p_{t,n}(T)e_{j+c_{k,n}-c_{t,n}})$.
\par\smallskip
-- If $r_{k}<h_{n}$, then $T^{c_{k,n}}f_{j}=4f_{j+c_{k,n}}$
since $j+c_{k,n}$ belongs to $I_{r_{1},\dots, r_{k}+1, \dots, r_{t}}$.
\par\smallskip
-- If $r_{k}=h_{n}$ and $k<t$, $j+c_{k,n}\in [r_{1}c_{1,n}+\dots+
(h_{n}+1)c_{k,n}+\dots+r_{t}c_{t,n}, r_{1}c_{1,n}+\dots+
(h_{n}+1)c_{k,n}+\dots+r_{t}c_{t,n}+\xi _{n}+1]$ which is contained in
the beginning of the \loi\ $[r_{1}c_{1,n}+\dots+
h_{n}c_{k,n}+\dots+r_{t}c_{t,n}+\xi _{n}+1, r_{1}c_{1,n}+\dots+
h_{n}c_{k,n}+(r_{k+1}+1)c_{k+1,n}+\dots+r_{t}c_{t,n}-1]$ if $r_{k+1}
<h_{n}$ (else we have to move over to the first $s$ with $r_{s}<h_{n}$ if
there is one, or else in the last \loi. We leave this to the reader). So
$||e_{j+c_{k,n}}||\ls 2^{-\frac {1}{2}\sqrt{c_{k+1,n}}}$. In the same way
$||e_{j+c_{k,n}-c_{t,n}}||\ls 2^{-\frac {1}{2}\sqrt{c_{k+1,n}}}$, and
since $||p_{t,n} (T)||\leq 2||T||^{d_{n}}\leq 2.2^{d_{n}}$ for instance,
and we get that $||T^{c_{k,n}}f_{j}||$ can be made arbitrarily small.
\par\smallskip
-- It is in the case where $r_{k}=h_{n}$ and $k=t$ that the
condition on $h_{n}$ appears, and this case has to be worked out
carefully: $T^{c_{k,n}}f_{j}=
\gamma _{n}^{-1}4^{1-|r|}(e_{j+c_{k,n}}-p_{k,n}(T)e_{j})$. As before $j+c_{k,n}
\in [r_{1}c_{1,n}+\dots+
(h_{n}+1)c_{k,n}, r_{1}c_{1,n}+\dots+ (h_{n}+1)c_{k,n}+\xi _{n}+1]$ which
is contained in the beginning of a \loi\
so $||e_{j+c_{k,n}}||\ls 2^{-\frac {1}{2}\sqrt{c_{1,n}}}$, the first
term in the expression of $T^{c_{k,n}}f_{j}$ can be made arbitrarily
small in norm, and thus is not a problem. Then we have to estimate the
quantity $||\gamma _{n}^{-1} 4^{1-|r|}\sum_{j\in I_{r_{1},\dots, h_{n}}}
x_{j}\,p_{k,n}(T)e_{j}||.$ By Lemma \ref{lem7},
\begin{eqnarray*}
p_{k,n}(T)e_{j}&=&p_{k,n}(T)(e_{j})_{1}+p_{k,n}(T)(e_{j})_{2}
=\Bigl(\sum_{l=1}^{k}
\sum_{s_{l}=0}^{r_{l}}
\gamma _{n} 4^{r_{1}+\dots r_{l-1}+(r_{l}-s_{l})-1}\\
&&
p_{l,n}(T)^{s_{l}}p_{l+1,n}(T)^{r_{l+1}}\dots p_{k,n}(T)^{h_{n}+1}
f_{r_{1}c_{1,n}+\dots+
r_{l-1}c_{l-1,n}+(r_{l}-s_{l})c_{l,n}+\alpha}\Bigr)\\
&+&
p_{1,n}(T)^{r_{1}}\dots p_{k,n}(T)^{h_{n}+1}e_{\alpha }
\end{eqnarray*}
where $j=r_{1}c_{1,n}+\dots+ h_{n}c_{k,n}+\alpha $ with $\alpha \in
[0,\xi _{n}]$.
The polynomial $$p_{l,n}(T)^{s_{l}}p_{l+1,n}(T)^{r_{l+1}}\dots p_{k,n}(T)^{h_{n}+1} $$ is of degree at most $(h_n+1)k_nd_n$
and its modulus
is less than $2^{s_l+r_{l+1}+\dots+h_n+1}$. When expanding the expression
$$p_{l,n}(T)^{s_{l}}p_{l+1,n}(T)^{r_{l+1}}\dots p_{k,n}(T)^{h_{n}+1}
f_{r_{1}c_{1,n}+\dots+
r_{l-1}c_{l-1,n}+(r_{l}-s_{l})c_{l,n}+\alpha} ,$$
two kind of terms appear:

-- multiples of
$ f_{r_{1}c_{1,n}+\dots+
r_{l-1}c_{l-1,n}+(r_{l}-s_{l})c_{l,n}+\alpha+u} $ for
$u\leq \xi_n$: this corresponds to ``small values'' of $u$,
for which $$T^u f_{r_{1}c_{1,n}+\dots+
r_{l-1}c_{l-1,n}+(r_{l}-s_{l})c_{l,n}+\alpha}=f_{r_{1}c_{1,n}+\dots+
r_{l-1}c_{l-1,n}+(r_{l}-s_{l})c_{l,n}+\alpha+u}.$$

-- multiples of
$T^uf_{r_{1}c_{1,n}+\dots+
r_{l-1}c_{l-1,n}+(r_{l}-s_{l})c_{l,n}+\alpha}$ for $\xi_n+1\leq
u\leq
\xi_n+(h_n+1)k_nd_n$ corresponding to ``large values'' of $u$.

For the first terms the norm can be directly computed, and for the second terms it suffices to notice that the expression of $T^uf_{r_{1}c_{1,n}+\dots+
r_{l-1}c_{l-1,n}+(r_{l}-s_{l})c_{l,n}+\alpha}$ involves only vectors $e_i$ for $i$ in the beginning of two \loi s between (c)-fans. Since the length of these intervals is roughly larger than $c_{1,n}$ which is much larger than $\xi_n$, $h_n$, $k_n$, $d_n$,
$||T^uf_{r_{1}c_{1,n}+\dots+
r_{l-1}c_{l-1,n}+(r_{l}-s_{l})c_{l,n}+\alpha}||$ is very small. This shows that
\begin{eqnarray*}
||
\sum_{j\in I_{r_{1},\dots, h_{n}}}x_{j}p_{k,n}(T)(e_{j})_{1}||
&\ls&\Bigl(
\sum_{l=1}^{k}
\sum_{s_{l}=0}^{r_{l}}\gamma_n 4^{r_1+\dots r_{l-1}+(r_{l}-s_{l})-1}
 2^{s_{l}+r_{l+1}+\dots+h_{n}+1)}\\
&&+2^{-\frac{1}{2}\sqrt{c_1,n}}
 \Bigr)
\;\Bigl(\sum_{j\in I_{r_{1},\dots, h_{n}}} |x_{j}|^{2}
 \Bigr)^{\frac {1}{2}}
\end{eqnarray*} so that
$$
||\gamma_n^{-1}4^{1-|r|}
\sum_{j\in I_{r_{1},\dots, h_{n}}}x_{j}p_{k,n}(T)(e_{j})_{1}||$$ is
approximately less than
$$
2 \,\Bigl(\sum_{s_{k}=0}^{h_{n}}2^{-s_{k}}+2^{-h_{n}}
 \sum_{s_{k-1}=0}^{h_{n}}2^{-s_{k-1}}+\dots
+
 2^{-(r_{2}+\dots+h_{n})}\sum_{s_{1}=0}^{h_{n}}2^{-s_{1}}+2^{-\frac{1}{2}\sqrt{c_1,n}}\Bigr)
 \Bigl(\sum_{j\in I_{r_{1},\dots, h_{n}}} |x_{j}|^{2}
 \Bigr)^{\frac {1}{2}}$$
 which is in turn less than
 \begin{eqnarray*}
4\;(1+k_{n}2^{-h_{n}}+2^{-\frac{1}{2}\sqrt{c_1,n}})\Bigl(\sum_{j\in I_{r_{1},\dots, h_{n}}} |x_{j}|^{2}
 \Bigr)^{\frac {1}{2}}
\leq 5\;\Bigl(\sum_{j\in I_{r_{1},\dots, h_{n}}} |x_{j}|^{2}
 \Bigr)^{\frac {1}{2}}
\end{eqnarray*}

if $h_{n}$ is large enough \wrt\ $k_{n}$. Then we estimate in the same way
$$
||\gamma _{n}^{-1}4^{1-|r|}
\sum_{j\in I_{r_{1},\dots, h_{n}}}x_{j}p_{k,n}(T)(e_{j})_{2}||$$
which is roughly less than $$\gamma _{n}^{-1}4^{1-|r|} 2^{|r|}
\Bigl(\sup_{u\leq \xi _{n}
+(h_{n}+1)k_{n}d_{n}} ||e_{u}||\Bigr)\;
\Bigl(\sum_{j\in I_{r_{1},\dots, h_{n}}}
|x_{j}|^2\Bigr)^{\frac {1}{2}}.
$$
Since $c_{1,n}$ is very large \wrt\ $(h_{n}+1)k_{n}d_{n}$, $||e_{u}||\ls
2^{-\frac {1}{2}\sqrt{c_{1,n}}}$ for $\xi _{n}+1\leq u \leq \xi _{n}
+(h_{n}+1)k_{n}d_{n}$. Recalling that $|r|=r_{1}+\dots+h_{n}$, we get
that
$$
||\gamma _{n}^{-1}4^{1-|r|}
\sum_{j\in I_{r_{1},\dots, h_{n}}}x_{j}(e_{j})_{2}||
\ls
C_{\xi _{n}}2^{-h_{n}}\Bigl(\sum_{j\in I_{r_{1},\dots, h_{n}}} |x_{j}|^{2}
 \Bigr)^{\frac {1}{2}}
$$
where $C_{\xi _{n}}$ depends only on $\xi _{n}$, and this is very small
if $h_{n}$ is large enough. Putting together all the estimates above, we
get that
$$||T^{c_{k,n}}\Bigl(\sum_{j\in I_{r_{1},\dots, h_{n}}}x_{j}
f_{j}\Bigr)||\leq 6\, \Bigl(\sum_{j\in I_{r_{1},\dots, h_{n}}}
|x_{j}|^2\Bigr)^{\frac {1}{2}}.$$
\par\smallskip
$\bullet$ It remains to study the case where $k>t$: for $j\in
I_{r_{1},\dots,r_{t}}$,
$$T^{c_{k,n}}f_{j}=
\gamma _{n}^{-1}4^{1-|r|}(e_{j+c_{k,n}}-p_{t,n}(T)e_{j+c_{k,n}
-c_{t,n}}).$$ Since $j+c_{k,n}
$ belongs to $ [r_{1}c_{1,n}+\dots+ r_{t}c_{t,n}+c_{k,n},
r_{1}c_{1,n}+\dots+ r_{t}c_{t,n}+c_{k,n}+\xi _{n}]$,
$f_{j+c_{k,n}}=\gamma _{n}^{-1}4^{-|r|}(e_{j+c_{k,n}}-p_{k,n}(T)e_{j}), $
so we have $$T^{c_{k,n}}f_{j}= 4f_{j+c_{k,n}}+
\gamma _{n}^{-1}4^{1-|r|}p_{k,n}(T)e_{j}-\gamma _{n}^{-1}4^{1-|r|}p_{t,n}(T)e_{j+c_{k,n}
-c_{t,n}}.$$ Now
$f_{j}=\gamma _{n}^{-1}4^{1-|r|}(e_{j}-p_{t,n}(T)e_{j-c_{t,n}})$
 so that
$$p_{k,n }(T)f_{j}=
\gamma _{n}^{-1}4^{1-|r|}(p_{k,n}(T)e_{j}-
p_{k,n}(T)p_{t,n}(T)e_{j-c_{t,n}}),$$ and $f_{j+c_{k,n}-c_{t,n}}=
\gamma _{n}^{-1}4^{1-|r|}(e_{j+c_{k,n}-c_{t,n}}-
p_{t,n}(T)e_{j-c_{t,n}})$ so that
$$p_{t,n }(T)f_{j+c_{k,n}-c_{t,n}}=\gamma _{n}^{-1}4^{1-|r|}(p_{t,n}(T)
e_{j+c_{k,n}-c_{t,n}}-p_{k,n}(T)p_{t,n}(T)e_{j-c_{t,n}}).$$ Hence
$T^{c_{k,n}}f_{j}=4f_{j+c_{k,n}}+p_{k,n}(T)f_{j}-p_{t,n}(T)f_{j+c_{k,n}
-c_{t,n}}$.
We then estimate $$||T^{c_{k,n}}\Bigl(\sum_{j\in
 I_{r_{1},\dots, r_{t}}}x_{j}
f_{j}\Bigr)||$$ as previously, writing $p_{k,n}(\zeta
)=\sum_{u=0}^{d_{n}} a_{u}^{(k)}\zeta ^{u}$:
\begin{eqnarray*}
&&||\sum_{u=0}^{d_{n}} a_{u}^{(k)}\sum_{j\in
 I_{r_{1},\dots, r_{t}}}x_{j}T^{u}f_{j}||\leq||
 \sum_{u=0}^{d_{n}}
a_{u}^{(k)}\sum_{\alpha =0}^{\xi _{n}-u}x_{r_{1}c_{1,n}+\dots+
r_{t}c_{t,n}+\alpha }f_{r_{1}c_{1,n}+\dots+ r_{t}c_{t,n}+\alpha+u }||\\
&&\qquad\qquad+||
 \sum_{u=0}^{d_{n}}
a_{u}^{(k)}\sum_{\alpha =\xi _{n}-u+1}^{\xi _{n}}x_{r_{1}c_{1,n}+\dots+
r_{t}c_{t,n}+\alpha }T^{u}f_{r_{1}c_{1,n}+\dots+ r_{t}c_{t,n}+\alpha}||
\end{eqnarray*}
and just as before the first term is less than $2
\Bigl(\sum_{j\in I_{r_{1},\dots, r_{t}}}
|x_{j}|^2\Bigr)^{\frac {1}{2}}$ while the second term is less than
$\varepsilon _{n}\Bigl(\sum_{j\in I_{r_{1},\dots, r_{t}}}
|x_{j}|^2\Bigr)^{\frac {1}{2}}$ with $\varepsilon _{n}$ arbitrarily
small.
\par\smallskip
It remains to put all the estimates together, and this finishes the
proof.
\end{proof}

We are now ready for the proof of Proposition \ref{prop5}:

\begin{proof}[Proof of Proposition \ref{prop5}]
For $x$ such that $\pi_{[0,\xi _{n}]}(x)=0$, let us decompose
$T^{c_{k,n}}x$ as
\begin{eqnarray*}
&&T^{c_{k,n}}x=
\pi_{[0,\xi _{n}]} (T^{c_{k,n}}(x))
+\sum_{l=n}^{+\infty }\pi_{[\xi _{l}+1,\xi _{l+1}]}(T^{c_{k,n}}x)\\
&&=\sum_{k=n}^{+\infty }
\pi_{[0,\xi _{n}]} (T^{c_{k,n}}(\pi_{[\xi _{k}+1,\xi _{k+1}]}(x)))
+\sum_{l=n}^{+\infty }\sum_{k=l-1}^{+\infty }
\pi_{[\xi _{l}+1,\xi _{l+1}]}\left(T^{c_{k,n}}(\pi_{[\xi _{k}+1,\xi _{k+1}]}(x))
\right).
\end{eqnarray*}
Indeed if the sequence $(\xi _{n})$ grows fast enough,
$T^{c_{k,n}}(\pi_{[\xi _{k}+1,\xi _{k+1}]}(x))$ for $k\geq n$  is
supported by $[0,\xi_{k+2}]$. So
\begin{eqnarray*}
||T^{c_{k,n}}x||^{2}&\leq & ||\sum_{k=n}^{+\infty }
\pi_{[0,\xi _{n}]} (T^{c_{k,n}}(\pi_{[\xi _{k}+1,\xi _{k+1}]}(x)))||^{2}\\
&+&\sum_{l=n}^{+\infty }||\sum_{k=l-1}^{+\infty }
\pi_{[\xi _{l}+1,\xi _{l+1}]}\left(T^{c_{k,n}}(\pi_{[\xi _{k}+1,\xi _{k+1}]}(x))
\right)||^{2}.
\end{eqnarray*}
Using the inequality $||a_{1}+\ldots+a_{j}||^{2}\leq
\sum_{i=1}^{j}2^{j+1-i} ||a_{i}||^{2}$ valid for every $j$-tuple
$(a_{1},\ldots, a_{j})$ of vectors of $H$, we get
\begin{eqnarray*}
||T^{c_{k,n}}x||^{2}&\leq &
\sum_{k=n}^{+\infty }2^{k+1-n}||
\pi_{[0,\xi _{n}]} (T^{c_{k,n}}(\pi_{[\xi _{k}+1,\xi _{k+1}]}(x)))||^{2}\\
&+&\sum_{l=n}^{+\infty }\sum_{k=l-1}^{+\infty }2^{k+2-l}||
\pi_{[\xi _{l}+1,\xi _{l+1}]}\left(T^{c_{k,n}}(\pi_{[\xi _{k}+1,\xi _{k+1}]}(x))
\right)||^{2}.
\end{eqnarray*}
This yields
\begin{eqnarray*}
||T^{c_{k,n}}x||^{2}&\leq & 2\,||
\pi_{[0,\xi _{n}]} (T^{c_{k,n}}(\pi_{[\xi _{n}+1,\xi _{n+1}]}(x)))||^{2}\\
&+&
\sum_{k=n+1}^{+\infty }2^{k+1-n}||
\pi_{[0,\xi _{n}]} (T^{c_{k,n}}(\pi_{[\xi _{k}+1,\xi _{k+1}]}(x)))||^{2}\\
&+&
\sum_{k=n-1}^{+\infty }||\sum_{l=n}^{k+1} 2^{k+2-l}
||\pi_{[\xi _{l}+1,\xi _{l+1}]}\left(T^{c_{k,n}}(\pi_{[\xi _{k}+1,\xi
_{k+1}]}(x))
\right)||^{2}.
\end{eqnarray*}
Now if $k\geq n+1$, $c_{k,n}<\xi _{k}/2$, so that by (1b')
\begin{eqnarray*}
||
\pi_{[0,\xi _{n}]} (T^{c_{k,n}}(\pi_{[\xi _{k}+1,\xi _{k+1}]}(x)))||&\leq
& ||
\pi_{[0,\xi _{k}]} (T^{c_{k,n}}(\pi_{[\xi _{k}+1,\xi _{k+1}]}(x)))||\\
&\leq& \delta _{k} || \pi_{[\xi _{k}+1,\xi _{k+1}]}(x)|| .
\end{eqnarray*}
If $n\leq l<k$, by (2b)
\begin{eqnarray*}
||\pi_{[\xi _{l}+1,\xi _{l+1}]}\left(T^{c_{k,n}}(\pi_{[\xi _{k}+1,\xi
_{k+1}]}(x))
\right)||& \leq &||
\pi_{[0,\xi _{k}]} (T^{c_{k,n}}(\pi_{[\xi _{k}+1,\xi _{k+1}]}(x)))||\\
&\leq& \delta _{k} || \pi_{[\xi _{k}+1,\xi _{k+1}]}(x)|| .
\end{eqnarray*}
If $k=l\geq n$, then by (2a)
\begin{eqnarray*}
||\pi_{[\xi _{k}+1,\xi _{k+1}]}\left(T^{c_{k,n}}(\pi_{[\xi _{k}+1,\xi
_{k+1}]}(x))
\right)||& \leq &
(1+\delta _{k}) || \pi_{[\xi _{k}+1,\xi _{k+1}]}(x)|| ,
\end{eqnarray*}
and lastly for $l=k+1$, $k\geq n-1$, by (3)
\begin{eqnarray*}
||\pi_{[\xi _{k+1}+1,\xi _{k+2}]}\left(T^{c_{k,n}}(\pi_{[\xi _{k}+1,\xi
_{k+1}]}(x))
\right)||& \leq & \delta _{k+1} || \pi_{[\xi _{k}+1,\xi _{k+1}]}(x)|| .
\end{eqnarray*}
Putting everything together yields that
\begin{eqnarray*}
||T^{c_{k,n}}x||^{2}& \leq & 2\,\delta _{n}^{2}||\pi_{[\xi _{n}+1,\xi
_{n+1}]}(x)||^{2} +\sum_{k=n}^{+\infty }
\left(\sum_{l=n}^{k+1}2^{k+2-l}\right)\delta _{k}^{2} \,||\pi_{[\xi _{k}+1,\xi
_{k+1}]}(x)||^{2}\\
&+&\sum_{k=n}^{+\infty }4\,(1+\delta _{k})^{2}||
\pi_{[\xi _{k}+1,\xi
_{k+1}]}(x)||^{2}\\
&+&\sum_{k=n}^{+\infty }2\,\delta _{k+1}^{2}||
\pi_{[\xi _{k}+1,\xi
_{k+1}]}(x)||^{2}.
\end{eqnarray*}
Since $\sum_{k=n}^{+\infty }||
\pi_{[\xi _{k}+1,\xi
_{k+1}]}(x)||^{2}\leq ||x||^2$, we get that if $\delta _{n}$ goes fast enough to zero then
 $||T^{c_{k,n}}x||\leq 100\, ||x||$
and we are done.
\end{proof}

The road to Theorem \ref{th1} is now clear.
\begin{proof}[Proof of Theorem \ref{th1}]
Let us choose for every $n\geq 1$ the family $(p_{k,n})_{1\leq k\leq
k_{n}}$ to form a $4^{-\xi _{n}}$ net of the closed ball of $\K_{n}[\zeta
]$ of radius $2$ (we take $d_{n}=n$ here). Then for every polynomial $q$
with $|q|\leq 2$ and for every $n$ greater than
the degree of $q$, there exists a $k\in [1,k_{n}]$ such
that $||p_{k,n}(T)-q(T)||\leq |p_{k,n}-q|\,.\, ||T||^{n}\leq
4^{-\xi _{n}}2^{n}\leq 2^{-n}$ if $||T||\leq 2$ for instance. Let us then estimate
for $x\in H$ the quantity $||T^{c_{k,n}}x-q(T)x||$:
\begin{eqnarray*}
||T^{c_{k,n}}x-q(T)x||&\leq& ||T^{c_{k,n}}(x-\pi_{[0,\xi
_{n}]}x)||+||T^{c_{k,n}}
\pi_{[0,\xi _{n}]}x-p_{k,n}(T)\pi_{[0,\xi _{n}]}x||\\
&+& ||p_{k,n}(T)\pi_{[0,\xi _{n}]}x-q(T)\pi_{[0,\xi _{n}]}x||+
||q(T)(x-\pi_{[0,\xi _{n}]}x)||\\
&\leq& 100\, ||x-\pi_{[0,\xi_{n}]}x || +\delta_n ||x||+2^{-n}||x||+||q(T)||\,||x-\pi_{[0,\xi
_{n}]}x||.
\end{eqnarray*}
Since $\delta _{n}$ and $||x-\pi_{[0,\xi_{n}]}x||$ go to zero as $n$ goes
to infinity, we see that $||T^{c_{k,n}}x-q(T)x||$ can be made arbitrarily
small. This implies that for every polynomial $q$ with $|q|\leq 2 $,
every $\varepsilon >0$ and every $x\in H$, there exists an integer $r$
 such that $||T^{r}x-p(T)x||<\varepsilon
$. Now if $p$ is any polynomial with $|p|\leq 2 ^{j}$ for some
nonnegative integer $j$, then for every $\varepsilon >0$ and every $x\in
H$ there exists an integer $r_{j} $ such that
$||T^{r_{j}}x-2^{-j}p(T)x||<\varepsilon 2^{-2j}$. Then $||2T^{r_{j}}x-
2^{-(j-1)}p(T)x||<\varepsilon 2^{-(2j-1)}$,
 and there exists an integer $r_{j-1}$
such that $||T^{r_{j-1}}x-2T^{r_{j}}x||<\varepsilon 2^{-(2j-1)}$. Hence
$||T^{r_{j-1}}x-2^{-(j-1)}p(T)x||<\varepsilon 2^{-2(j-1)}$. Continuing in
this fashion, we obtain an integer $r_{0}$ such that
$||T^{r_{0}}x-p(T)x||<\varepsilon $. Finally, notice that $e_{0}$ is a
cyclic vector for $T$ by construction, so it is in fact a \hy\ vector.
\end{proof}

\begin{remark}
For the proof of Theorem \ref{th1}, one actually does not need the full
complexity of the (c)-fan as presented here. It would be sufficient to
consider at each step $n$ only one polynomial $p_{n}$ and its associated fan
consisting of the intervals $[rc_{n}, rc_{n}+\xi_{n}]$, $r\in[0,h_{n}]$.
But we will need to be able to handle several polynomials $p_{1,n}, \ldots, p_{k_{n},n}$
at each step in the proof of Theorem \ref{th2bis}, and this is why we
present the complete (c)-fan already here.
\end{remark}

\section{Exhibiting hypercyclic vectors: the role of the (b)-fan}
Let $x$ be any non-zero vector of $H$ with $||x||\leq 1$. The starting
point of the proofs in \cite{R2} or \cite{R3} that $x$ must be cyclic for
$T$ is the following argument: consider the space $F_{\xi _{n}}=
\textrm{sp}[e_{j} \textrm{ ; } 0\leq j\leq \xi _{n}]
=\textrm{sp}[f_{j} \textrm{ ; } 0\leq j\leq \xi _{n}]$
and the \op\ $T_{\xi _{n}}$ on it which is the truncated forward shift on
$F_{\xi _{n}}$: $T_{\xi _{n}}e_{j}=e_{j+1}$ for $j<\xi _{n}$ and $T_{\xi
_{n}}e_{\xi _{n}}=0$. Write
$$\pi_{[0,\xi _{n}]}x=
\sum_{j=r_{n}}^{\xi _{n}}\alpha _{j}^{(n)}e_{j} =
\sum_{j=r_{n}}^{\xi _{n}}e_{j}^{*(n)}(x)e_{j}\;\;\textrm{ with
}\;\;\alpha _{r_{n}}
^{(n)}=e _{r_{n}}
^{*(n)}(x)\not =0$$ (since $x$ is non-zero, this is always possible if $n$
is large enough). The functionals $e_{j}^{*(n)}$, $j=0,\dots, \xi _{n}$
are the coordinate functionals \wrt\ the basis $(e_{j})_{j=0,\dots, \xi
_{n}}$ of $F_{\xi _{n}}$.
 Then it is easy to see that the linear orbit
of $\pi_{[0,\xi _{n}]}x$ under $T_{\xi _{n}}$ is $\textrm{sp}[ e_{j}
\textrm{ ; } r_{n}\leq j\leq \xi _{n}]$. Hence if one of the vectors
$e_{j}$, $r_{n}\leq j\leq \xi _{n}$ is sufficiently close to $e_{0}$ for
instance, then there exists a polynomial $p$ of degree less than $\xi
_{n}$ such that $||p(T_{\xi _{n}})\pi_{[0,\xi _{n}]}x
-e_{0}||$ is very small, and the difficulty is to estimate the
tail terms in order to show that one must have $||p(T)x-e_{0}||$ very
small too. The obvious way to start this is to estimate
$||p(T)\pi_{[0,\xi _{n}]}x-e_{0}||$: if $p(\zeta)=\sum_{u=0}^{\xi _{n}}
a_{u}\xi ^{u} $ and $0\leq j\leq \xi _{n}$, $p(T)e_{j}=\sum_{u=0}^{\xi
_{n}}a_{u}e_{j+u}$ and $p(T_{\xi _{n}})e_{j}=\sum_{u=0}^{\xi
_{n}-j}a_{u}e_{j+u}$
 so that $(p(T)-p(T_{\xi _{n}}))e_{j}=\sum_{u=\xi _{n}-j+1}^{\xi _{n}}
 a_{u}e_{j+u}$. Hence
\begin{eqnarray*}
&&||(p(T)-p(T_{\xi _{n}}))\pi_{[0,\xi _{n}]}x||=||\sum_{j=0}^{\xi _{n}}
\alpha _{j}^{(n)}\sum_{u=\xi _{n}-j+1}^{\xi _{n}}a_{u}e_{j+u}||\\
&&\qquad\qquad=||\sum_{u=0}^{\xi _{n}}a_{u}\sum_{j=\xi _{n}-u+1}^{\xi
_{n}}
\alpha _{j}^{(n)}e_{j+u}||\leq |p|\, C_{\xi _{n}}\,\sup_{\xi _{n}+1
\leq u\leq 2\xi _{n}}||e_{j+u}||\, ||x||.
\end{eqnarray*}

The quantity $\sup_{\xi _{n}+1
\leq u\leq 2\xi _{n}}||e_{j+u}||$ is very small compared to
 $\xi _{n}$ ($\ls 2^{-\frac {1}{2}\sqrt{c_{1,n}}}$) with our actual
 construction, so if $|p|$ is controlled by a constant depending
 only on $\xi _{n}$, the quantity
 $||(p(T)-p(T_{\xi _{n}}))\pi_{[0,\xi _{n}]}x||$ will be very
 small. The following fact is easy to prove, see the forthcoming Lemma \ref{lem1b} for a more precise estimate:

\begin{fact}
Let $\varepsilon _{\xi _{n}}$ be a positive constant depending only on
$\xi _{n}$. There exists a constant $C_{\xi _{n}}$ depending only on
$\xi_{n}$ such that for every $x\in H$, $||x||\leq 1$, such that $|\alpha
_{r_{n}}^{(n)}|\geq \varepsilon _{\xi
_{n}}$, for every $j\in [r_{n},\xi _{n}]$,
there exists a polynomial $p$ of degree less than $\xi _{n}$
 with $|p|\leq C_{\xi _{n}}$ such that $$||p(T_{\xi _{n}})
 \pi_{[0,\xi _{n}]}x-e_{j}||\leq \frac {1}{\xi _{n}} \cdot$$
\end{fact}

Hence with our informal assumptions $||p(T_{\xi _{n}})
 \pi_{[0,\xi _{n}]}x-e_{0}||$ is very small. The next step is to
 control the tail $||p(T)(x-\pi_{[0,\xi _{n}]}x)||$, and for
 this a natural idea is to use the (c)-fan:
$$||T^{c_{k,n}}x-e_{0}||\leq 100\, ||x-\pi_{[0,\xi _{n}]}x||+||
(T^{c_{k,n}}
-p_{k,n}(T))\pi_{[0,\xi _{n}]}x||+||p_{k,n}(T)
\pi_{[0,\xi _{n}]}x-e_{0}||,
$$ and then to approximate the polynomial $p$ by some $p_{k,n}$
in such a way that $|p-p_{k,n}|\leq 4^{-\xi _{n}}$ for instance. But here
we run into a difficulty: $|p_{k,n}|\leq 2$ for every $n$ and $1\leq
k\leq k_{n}$, while $|p|$ may be very large. Since the proof of the
uniform estimates for the (c)-fan really requires a uniform bound on the
quantities $|p_{k,n}|$, we have to modify the construction so as to
ensure the existence of  a polynomial $q$ with $|q|$ small such that
$||(p(T)-q(T))\pi_{[0,\xi _{n}]}x||$ is very small, and then we will be
able to approximate $q $ by $p_{k,n}$. The \emph{(b)-fan} is introduced exactly
for this purpose: we will see that it ensures that
$$\left|\left|\Bigl(\frac {T^{b_{n}}}{b_{n}}-I\Bigr)T(\pi_{[0,\xi _{n}]}x)
\right|\right|\leq \frac {1}{b_{n}}||x||$$
where $b_{n}$ is very large \wrt\ $\xi _{n}$. Then if the polynomial $p$ can be written as
$p(\zeta)=\zeta p_0(\zeta)$,
$$||\Bigl(p(T)\frac {T^{b_{n}}}{b_{n}}-p(T)\Bigr)\pi_{[0,\xi _{n}]}x||
=||\Bigl(p_0(T)\frac {T^{b_{n}}}{b_{n}}-p_0(T)\Bigr)T(\pi_{[0,\xi _{n}]}x)||
\leq |p|2^{\xi _{n}}\frac {1}{b_{n}}||x||\leq \frac {C_{\xi _{n}}}{b_{n}}
||x||$$ will be extremely small, while $|q(\zeta )|=|p(\zeta )\frac
{\zeta ^{b_{n}}}{b_{n}}|
\leq \frac {C_{\xi _{n}}}{b_{n}}$ will be less than $1$ if $b_{n}$
is large enough. Then one has to approximate $q$ by some polynomial
$p_{k,n}$, but here another difficulty appears: the degree of $q$ is not
bounded by $\xi _{n}$ anymore, but by $\xi
_{n}+b_{n}$, which is much larger, and so one has to modify the
fan constructed in Section $2$, which we from now on call the
\emph{(c)-fan},
 accordingly. We now describe in more details the (b)-fan and the
modifications of the (c)-fan.

\subsection{Construction of the (b)-fan, modification of the
(c)-fan} The (b)-fan consists of $\xi _{n}$ intervals, which are
introduced between $\xi _{n}+1$ and the (c)-fan, and depend on a number
$b_{n}$ chosen extremely large \wrt\ $\xi _{n}$. The intervals of the
(b)-fan are the intervals $[r(b_{n}+1), rb_{n}+\xi
_{n}]$, $r=1,\dots, \xi_{n} $, and for $j$ in one of these
intervals $f_{j}$ is defined as
$$f_{j}=e_{j}-b_{n}e_{j-b_{n}}.$$ The intervals between the
(b)-working intervals are \loi s and they are of length approximately
$b_{n}$, but we modify slightly the definition of $\lambda _{j}$ for $j$
in a (b)-\loi, just for convenience's sake: for $j\in [rb_{n}+\xi _{n}+1,
(r+1)b_{n}-1]$, $$\lambda _{j}= 2^{(\frac {1}{2}{b_{n}}+rb_{n}+\xi
_{n}+1-j)/\sqrt{b_{n}}}$$ and for $j\in [\xi _{n}+1, b_{n}]$,
$$\lambda _{j}=2^{(\frac {1}{2}{b_{n}}+\xi _{n}+1-j)/\sqrt{b_{n}}}$$
(instead of using the length of the \loi\ in the definition we use
$b_{n}$ which is of the same order of magnitude). The (b)-fan terminates
at the index $\nu _{n}=\xi
_{n}(b_{n}+1)$.
\par\smallskip
We have not yet proved that $T$ remains bounded with this addition of the
(b)-fan, but admitting this for the time being, we can see immediately
that $T^{b_{n}}/b_{n}$ is very close to the identity \op\ on vectors of $F_{\xi
_{n}}$ of the form $x=\sum_{j=1}^{\xi _{n}} \alpha _{j}^{(n)}e_{j}$, which was one of the reasons for introducing this (b)-fan:

\begin{fact}\label{fact12}
For every $x$ supported in $[0,\xi _{n}]$,
$$\left|\left|\Bigl(\frac {T^{b_{n}}}{b_{n}}-I\Bigr)T(x)\right|\right|
\leq \frac {C_{\xi _{n}}}{b_{n}}||x||.$$
\end{fact}

\begin{proof}
Write $x=\sum_{j=0}^{\xi _{n}} \alpha _{j}^{(n)}e_{j}$. Then
\begin{eqnarray*}
\Bigl(\frac {T^{b_{n}}}{b_{n}}-I\Bigr)T(x)&=&\sum_{j=0}^{\xi _{n}} \alpha _{j}^{(n)}
\Bigl(\frac {1}{b_{n}}e_{j+b_{n}+1}-e_{j+1}\Bigr)\\
&=&
\sum_{j=0}^{\xi _{n}-1} \alpha _{j}^{(n)}
\frac {1}{b_{n}}f_{j+b_{n}+1}+
\alpha _{\xi_n}^{(n)}
\Bigl(\frac {1}{b_{n}}e_{b_{n}+\xi_n+1}-e_{\xi_n+1}\Bigr)
\end{eqnarray*}
because $f_{j+b_{n}}=e_{j+b_{n}}-b_{n}e_{j}$
for $j\in [1,\xi _{n}]$. Since $||e_{b_{n}+\xi_n+1}||
\ls 2^{-\frac{1}{2}\sqrt{b_n}}$ and
$||e_{\xi_n+1}||
\ls 2^{-\frac{1}{2}\sqrt{b_n}}$, the estimate of Fact \ref{fact12} follows.
\end{proof}

Fact \ref{fact12} motivates the introduction of the interval $[b_{n}+1,
b_{n}+\xi _{n}]$, but the role of the ``shades'' $[r(1+b_{n}),rb_{n}+\xi
_{n}]$, $r=2,\dots, \xi _{n}$ which appear afterwards is still obscure at
this stage of the construction. The motivation for this will be explained
later on.
\par\smallskip
Let us now explain why we have to modify the (c)-fan: using the previous
construction, we have seen that if $q(\zeta )=\frac {\zeta ^{b_{n}}}{
b_{n}}p(\zeta )=\frac {\zeta ^{b_{n}+1}}{
b_{n}}p_0(\zeta )$, then $|q|<1$, $q$ is of degree less than $b_{n}+\xi
_{n}$, so in particular less than $\nu _{n}$, and
$||q(T)\pi_{[0,\xi_{n}]}x-e_{0}|| $ is very small. Our goal is now to
approximate $q$ for $|\,.\,|$ by some polynomial $p_{k,n}$. With our
actual construction this is impossible, because the degree of $p_{k,n}$
is too large: if for instance we try to estimate $||Tf_{c_{k,n}+\xi
_{n}}||=||\gamma _{n}^{-1}(e_{c_{k,n}+\xi _{n}+1}
-p_{k,n}(T)e_{\xi _{n}+1})||$, the upper bound we get involves
$$\gamma_{n}^{-1}
\sup_{0\leq j\leq \xi _{n}+b_{n}}||e_{\xi _{n}+1+j}||$$
which is by no means small. So we have to increase the length of the
(c)-working intervals from $\xi _{n}$ to $\nu _{n}$ (recall that
$\nu_{n}=\xi _{n}(b_{n}+1)$ is the index of the last (b)-working interval), and to chose the
family $(p_{k,n})$ as a $4^{-\nu _{n}}$ net of the unit ball of the set
$\K_{\nu_{n} }[\zeta ]$ of polynomials of degree less than $\nu _{n}$. The (c)-fan
starts at $c_{1,n}$ very large \wrt\ $\nu
_{n}$, and the (c)-working intervals are
$[r_{1}c_{1,n}+\dots+c_{k_{n}}r_{k_{n},n},
r_{1}c_{1,n}+\dots+c_{k_{n}}r_{k_{n},n}+\nu _{n}]$, $r_{i}\in [0,h_{n}]$
for $i=1,\dots,k_{n}$. With this definition, the analogue of Fact
\ref{fact4} will be:

\begin{fact}\label{fact13}
Let $\delta _{n}$ be any small positive number. If $\gamma _{n}$ is small
enough, then for every vector $x$ supported in $[0,\nu_{n}]$ and every
$1\leq k\leq k_{n}$, $$||T^{c_{k,n}}x-p_{k,n}x||\leq \delta  _{n}||x||.$$
\end{fact}

Notice that $e_0$ remains hypercyclic with this introduction of the (b)-fan.

\subsection{Boundedness of $T$, estimates on $T^{c_{k,n}}$}
We first have to check that $T$ is still bounded with these
modifications. This will follow from Proposition \ref{prop11} below,
which is the analogue of our previous Proposition \ref{prop6}:

\begin{proposition}\label{prop11}
Let $(\delta _{n})_{n\geq 0}$ be a decreasing sequence of positive
numbers going to zero very fast. The  vectors $f_{j} $ can be constructed
so that for every $n\geq 0$, assertion (1) below holds true:

\begin{enumerate}
\item[(1)] if $x$ is supported in the interval $[\nu _{n}+1, \nu
_{n+1}]$, then

\begin{enumerate}
\item[(1a)] $||\pi_{[\nu _{n}+1, \nu
_{n+1}]}(Tx)||\leq (1+\delta _{n}) ||x||$

\item[(1b)] $||\pi_{[0, \nu
_{n}]}(Tx)||\leq \delta _{n} ||x||$.
\end{enumerate}
\end{enumerate}
\end{proposition}

\begin{proof}
We just outline the points which are different from the proof of
Proposition \ref{prop6}.
\par\smallskip
$\bullet$ If $j=rb_{n}+\xi _{n}$, $Tf_{j}=e_{rb_{n}+\xi _{n}+1}-
b_{n}e_{(r-1)b_{n}+\xi _{n}+1}$. Since $rb_{n}+\xi _{n}+1$
 and $(r-1)b_{n}+\xi _{n}+1$ are the endpoints of \loi s of
 length at least roughly $b_{n}$,
 $||e_{rb_{n}+\xi _{n}+1}||\ls 2^{-\frac {1}{2}\sqrt{b_{n}}}$ and
$||e_{(r-1)b_{n}+\xi _{n}+1}||\ls 2^{-\frac {1}{2}\sqrt{b_{n}}}$, so
$||Tf_{j}||$ can be made arbitrarily small.
\par\smallskip
$\bullet$ If $j=r(b_{n}+1)-1$ is the endpoint of a \loi\ of type (b),
$Tf_{j}=\lambda _{r(b_{n}+1)-1}e_{r(b_{n}+1)}$. Now we have a formula
for $e_{r(b_{n}+1)}$ similar to the one of Lemma \ref{lem7}, but much
simpler since we go down a one-dimensional lattice, not a
multi-dimensional one:
$$e_{r(b_{n}+1)}=\sum_{l=0}^{r-1}b_{n}^{l}f_{(r-l)b_{n}+r}+b_{n}^{r
e_{r}}.$$ Hence $||e_{r(b_{n}+1)}||\ls b_{n}^{\xi _{n}}C_{\xi _{n}}$
where $C_{\xi _{n}}$ depends only on $\xi _{n}$. Since $\lambda
_{r(b_{n}+1)-1}\ls 2^{-\frac {1}{2}\sqrt{b_{n}}}$, $||Tf_{j}||$ can be
made very small too.
\par\smallskip
$\bullet$ The proof of the estimates for
$||Tf_{r_{1}c_{1,n}+\dots+r_{k_{n}}c_{k_{n},n}-1}||$ and
$||Tf_{r_{1}c_{1,n}+\dots+r_{k_{n}}c_{k_{n},n}+\nu _{n}}||$ are exactly
the same as in Proposition \ref{prop6}, except that $||e_{\nu _{n}+1}||$
is now involved instead of $||e_{\xi
_{n}+1}||$:
$||e_{\nu _{n}+1}||\ls 2^{-\frac {1}{2}\sqrt{c_{1,n}}}$ and everything
works as previously.
\end{proof}

\begin{corollary}
For any $\varepsilon >0$ one can make the construction so that $T$ is
bounded on $H$ with $||T||\leq 1+\varepsilon $.
\end{corollary}

Then Proposition \ref{prop10} becomes

\begin{proposition}\label{prop14}
Let $(\delta _{n})_{n\geq 0}$ be a decreasing sequence of positive
numbers going to zero very fast. The  vectors $f_{j} $ can be constructed
so that for every $n\geq 0$, assertion (2) below holds true:

\begin{enumerate}
\item[(2)] for any
vector $x$  supported in the interval $[\nu _{n}+1, \nu
_{n+1}]$ and for any $1\leq k \leq k_{n}$,
\begin{enumerate}
\item[(2a)] $||\pi_{[\nu _{n}+1, \nu
_{n+1}]}(T^{c_{k,n}}x)||\leq 4 ||x||$

\item[(2b)] $||\pi_{[0, \nu
_{n}]}(T^{c_{k,n}}x)||\leq \delta _{n} ||x||$
\end{enumerate}

\item[(3)] for any $x$ supported in the interval $[0, \nu
_{n}]$ and any $m<\nu _{n}/2$,
\begin{enumerate}
\item[$\;$]$||\pi_{[\nu _{n}+1, \nu
_{n+1}]}(T^{m}x)||\leq \delta _{n} ||x||$.
\end{enumerate}
\end{enumerate}
\end{proposition}

The same argument which is used for the proof of Proposition \ref{prop5}
shows that

\begin{proposition}\label{prop15}
For every $n\geq 1$, every $1\leq k \leq k_{n}$ and every $x\in H$ such
that $\pi_{[0,\nu _{n}]}x=0$, $||T^{c_{k,n}}x||\leq 100\,||x||$. In other
words, $$||T^{c_{k,n}}(x-\pi_{[0,\nu _{n}]}x)||\leq \,100
||x-\pi_{[0,\nu_{n}]}x|| $$ for every $x\in H$.
\end{proposition}

\begin{proof}[Proof of Proposition \ref{prop14}]
The proof is virtually the same, except that one has additionally to
investigate the quantities $T^{c_{k,n}}f_{j}$ for $j\in [\xi _{n}+1,\nu
_{n+1}]$. This involves no difficulty:
\par\smallskip
$\bullet$ If $j$ and $j+c_{k,n}$ belong to the same \loi\ ($[\xi
_{n+1}+1,b_{n+1}]$ for instance), $T^{c_{k,n}}f_{j}=\lambda _{j}/
\lambda _{j+c_{k,n}}f_{j+c_{k,n}}$, and $
\lambda _{j}/
\lambda _{j+c_{k,n}}\ls 2^{\frac {1}{2}\sqrt{(c_{k,n}/b_{n+1})}}$ which can
be made arbitrarily close to $1$.
\par\smallskip
$\bullet$ If $j$ belongs to a \loi\ ending at the point $r(b_{n+1}+1)-1$
and $j+c_{k,n}$ belongs to the working interval
$[r(b_{n+1}+1),rb_{n+1}+\xi _{n+1}]$, then $T^{c_{k,n}}f_{j}
=T^{\alpha }e_{r(b_{n+1}+1)}$ with $0\leq \alpha \leq c_{k,n}$,
so
$$||T^{c_{k,n}}f_{j}||\leq \lambda _{j}\, ||T||^{c_{k,n}}
||e_{r(b_{n+1}+1)}||\ls \lambda _{j} \, 2^{c_{k,n}}
 \, b_{n+1}^{\xi _{n+1}}
C_{\xi_{n+1}}.$$ Now $r(b_{n+1}+1)-c_{k,n}\leq j\leq r(b_{n+1}+1)-1$
 and since
 $b_{n+1}$ is very large \wrt\ $c_{k,n}$, $\lambda _{j}\ls
 2^{-\frac {1}{2}\sqrt{b_{n+1}}}$, and thus $||T^{c_{k,n}}f_{j}||$ can be
 made very small.
 \par\smallskip
$\bullet$ The argument is exactly the same when $j$ belongs to a working
interval of the (b)-fan, and we omit it.
\end{proof}

\subsection{Estimates on $T^{b_{n}+1}$, construction of some hypercyclic vectors}
We begin this section by a result showing that if the $e_{0}$-coordinate
of $\pi_{[0,\xi _{n}]}x$ is not too small for infinitely many $n$'s,
then $x$ must be \hy. Though not strictly necessary for the proof of
Theorem \ref{th2}, this result shows the main idea of the proof, and
will allow us to prove easily that $HC(T)^{c}$ is Haar null, so we
include it.

\begin{proposition}\label{prop16}
Let $x\in H$, $||x||\leq 1$, be a vector satisfying the following
assumption:
\begin{enumerate}
\item [(*)] for infinitely many $n$'s, $|e_{0}^{*(n)}(x)|\geq
2^{-n}$, where $\pi_{[0,\xi _{n}]}x=\sum_{j=0}^{\xi _{n}} e _{j}
^{*(n)}(x)e_{j}$.
\end{enumerate}
Then $x$ is \hy\ for $T$.
\end{proposition}

\begin{proof}
By Fact \ref{fact4}, there exists
for every $n\geq 1$ a constant $C_{\xi _{n}} $ such that for
every $y\in F_{\xi _{n}}$ of the form
 $y=\sum_{j=0}^{\xi _{n}}
e_{j}^{*(n)}(y)e_{j}$ with $|e_{0}^{*(n)}(y)|\geq 2^{-n}$, there exists a
polynomial $p$ of degree less than $\xi _{n}$ with $|p|\leq C_{\xi _{n}}$
and such that $\zeta $ divides $p(\zeta )$ which has the property that
 $$||p(T_{\xi _{n}})y-e_{1}||\leq \frac {1}{\xi _{n}}
\cdot$$
Write $p(\zeta)=\zeta p_0(\zeta)$.
When $x$ satisfies (*), we choose an $n$ such that
$|e_{0}^{*(n)}(x)|\geq 2^{-n}$ and apply this to the vector
$y=\pi_{[0,\xi _{n}]}x$. If
 $p$ is the polynomial satisfying the above-mentioned properties,
  then we
have seen that
$$||p_0(T)T(\pi_{[0,\xi _{n}]}x)-e_{1}||\leq \frac {2}{\xi _{n}}
$$ since $||(p(T_{\xi _{n}})-p(T))\pi_{[0,\xi _{n}]}x||\ls
C_{\xi _{n}}\sup_{\xi _{n}+1\leq j\leq 2\xi _{n}}||e_{j}||
\ls C_{\xi _{n}}
2^{-\frac {1}{2}\sqrt{b_{n}}}$. We now have to make the modulus of the
polynomial small, so we take $q(\zeta )=\frac {\zeta ^{b_{n}}}{b_{n}} p
(\zeta )=\frac {\zeta ^{b_{n}+1}}{b_{n}} p_0
(\zeta )$: the degree of $q$ is less than $\xi _{n}+b_{n}$, $|q|<1$, and
by Fact \ref{fact12}
$$||q(T)\pi_{[0,\xi_{n} ]}x-e_{1}||\leq \frac {3}{\xi _{n}}
\cdot$$ Let now $k\leq k_{n}$ be such that $|q-p_{k,n}|\leq 4^{-\nu
_{n}}$: then $||q(T)-p_{k,n}(T)||\leq 4^{-\nu _{n}}||T||^{d}$
where $d$ is the degree of $q-p_{k,n}$, so $||q(T)-p_{k,n}(T)||\leq
2^{-\nu _{n}}$ for instance. Hence
$$||p_{k,n}(T)\pi_{[0,\xi_{n} ]}x-e_{1}||\leq \frac {4}{\xi _{n}}
\cdot$$ This yields that
\begin{eqnarray*}
||T^{c_{k,n}}x-e_{1}||&\leq&||(T^{c_{k,n}}-p_{k,n}(T))\pi_{[0,\nu
_{n}]}x||\\
&+& ||p_{k,n}(T)\pi_{[\xi _{n}+1,\nu_{n} ]}x||+||p_{k,n}(T)
\pi_{[0,\xi _{n}]}
x-e_{1}||\\ &\leq& \frac {5}{\xi _{n}}+||p_{k,n}(T)\pi_{[\xi _{n}+1,\nu
_{n}]}x||\\
&\leq& \frac {6}{\xi _{n}}+||q(T)\pi_{[\xi _{n}+1,\nu
_{n}]}x||
\end{eqnarray*}
and the difficulty which remains is to estimate the last term. This is
here that we use the fact (which may look a bit strange) that we have
approximated $e_{1}$ and not $e_{0}$, as well as the shades of the
(b)-fan: since $\zeta $ divides $p(\zeta )$ (because we approximate
$e_{1}$), $q$ can be written as $q(\zeta )=\frac {1}{b_{n}}\zeta
^{b_{n}+1} p_{0}(\zeta )$ with $|p_{0}|\leq C_{\xi _{n}} $ and the degree
of $p_{0}$ less than $\xi _{n}-1$. Hence
\begin{eqnarray*}
||q(T)\pi_{[\xi _{n}+1,\nu
_{n}]}x||=||\frac {1}{b_{n}}T ^{b_{n}+1}
p_{0}(T)\pi_{[\xi _{n}+1,\nu
_{n}]}x||\leq \frac {1}{b_{n}}C_{\xi _{n}} 2^{\xi _{n}}||T^{b_{n}+1}
\pi_{[\xi _{n}+1,\nu
_{n}]}x||.
\end{eqnarray*}
And now the shades of the (b)-fan have been introduced exactly so as to
ensure that

\begin{lemma}\label{lem17}
For every $n\geq 1$ and every $x\in H$,
$$||T^{b_{n}+1}
\pi_{[\xi _{n}+1,\nu
_{n}]}x||\leq 2||x||.$$
\end{lemma}

Lemma \ref{lem17} allows us to conclude immediately the proof of
Proposition \ref{prop16}:
$$||T^{c_{k,n}}x-e_{1}||\leq \frac {7}{\xi _{n}},$$ and hence $e_{1}$
belongs to the closure of the orbit of $x$. Since $e_{0}$ is \hy\ for
$T$, $e_{1}$ is too, and hence $x$ is hypercyclic.
\end{proof}

\begin{remark}
The condition $|e_{0}^{*(n)}(x)|\geq 2^{-n}$ can obviously be replaced
by any condition of the form $|e_{0}^{*(n)}(x)|\geq \varepsilon _{\xi
_{n}}$, where $\varepsilon _{\xi _{n}}$ is a small number depending
only on $\xi _{n}$.
\end{remark}

It remains to prove Lemma \ref{lem17}.

\begin{proof}[Proof of Lemma \ref{lem17}]
As in the preceding proofs, we must distinguish several cases.
\par\smallskip
$\bullet$ If $j\in [r(b_{n}+1),rb_{n}+\xi _{n}]$, $r=1,\dots, \xi
_{n}-1$, $T^{b_{n}+1}f_{j}=e_{j+b_{n}+1}-b_{n}e_{j+1}$, and $j+b_{n}+1
\in [(r+1)(b_{n}+1),(r+1)b_{n}+\xi _{n}]$, so if $j<rb_{n}+\xi
_{n}$, $T^{b_{n}+1}f_{j}=f_{j+b_{n}+1}$. If
 $j=rb_{n}+\xi
_{n}$,
$T^{b_{n}+1}f_{j}=e_{(r+1)b_{n}+\xi _{n}+1}- b_{n}e_{rb_{n}+\xi _{n}+1}$.
Since $||e_{rb_{n}+\xi _{n}+1}||\ls 2^{-\frac {1}{2}\sqrt{b_{n}}}$ for
$r=1,\dots,\xi _{n}$, $||T^{b_{n}+1}f_{j}||$ is very small.
\par\smallskip
$\bullet$ If $j=\xi _{n}(b_{n}+1)$, $T^{b_{n}+1}f_{j}=e_{ (\xi
_{n}+1)(b_{n}+1)}-b_{n}e_{\xi _{n}(b_{n}+1)+1}$ so
 $||T^{b_{n}+1}f_{j}||$
is very small.
\par\smallskip
$\bullet$ If $j\in [rb_{n}+\xi _{n}+1, (r+1)(b_{n}+1)-1]$, $r=1,\dots,
\xi _{n}-2$, $T^{b_{n} +1}f_{j}=\lambda _{j}e_{j+b_{n}+1}$ and $j+b_{n}+1
\in [(r+1)b_{n}+\xi _{n} +2, (r+2)(b_{n}+1)-1] $ which is contained in
the \loi\ $[(r+1)b_{n}+\xi _{n} +1, (r+2)(b_{n}+1)-1]$. So
$T^{b_{n}+1}f_{j}=\lambda _{j}/\lambda _{j+b_{n} +1}e_{j+b_{n}+1}$. Now a
straightforward computation shows that ${\lambda _{j}}/{\lambda
_{j+b_{n}+1}}=2^{
 {1}/{\sqrt{b_{n}}}}$ which is less than $2$ if $b_{n}$ is
sufficiently large. It is at this point that we use the fact that the
definition of the coefficients $\lambda _{j}$ for $j$ in a (b)-\loi\
involves directly $b_{n}$, and not the length of the interval. If $r=\xi
_{n}-1$, then $j+b_{n}+1$ belongs to the beginning of the \loi\ $[\nu
_{n}+1, c_{1,n}-1]$, so $\lambda _{j+b_{n}+1}\gs 2^{\sqrt{ c_{1,n}}}$ so
$\lambda _{j}/\lambda _{j+b_{n}+1}\ls 2^{\frac {1}{2}\sqrt{b_{n}}}
2^{-\sqrt{c_{1,n}}}$ is very small.
\par\smallskip
$\bullet$ If $j\in [\xi _{n}+1,b_{n}]$, $j+b_{n}+1\in [b_{n}+\xi _{n}+1,
2b_{n}]$, so again $T^{b_{n}+1}f_{j}=2^{1/\sqrt{b_{n}}}f_{j+b_{n}+1}$.
\end{proof}

\subsection{The set $HC(T)^{c}$ is Haar null}
If $M$ is any positive integer, let $E_{M}$ be the set of vectors
$
x\in H$ such that $||x||\leq M$ and there exists an $n_{0}$ such that for
every $n\geq n_{0}$, $ |e_{0}^{*(n)}(x) |\leq 2^{-n}M$. Then
$$HC(T)^{c}\subseteq \bigcup_{M=1}^{+\infty }E_{M}.$$ Indeed if $x$ is a nonzero
vector not in
 $HC(T)$, then $x/||x||$ does not satisfy assumption (*) of
 Proposition \ref{prop16}, so there exists an $n_{0}$ such that
 for every $n\geq n_{0}$,
 $|e_{0}^{*(n)}(x/||x||)|\leq 2^{-n}$, i.e. $|e_{0}^{*(n)}(x)|\leq 2^{-n}
 ||x||$. Hence if $M\geq ||x||$, $x$ belongs to $E_{M}$. Since
 the union of countably many Haar null sets is Haar null, it
 suffices to show that each $E_{M}$ is Haar null. There are
 different ways of proving this. A first option is to use a
 result of Matouskova \cite{Ma} that every closed convex subset
 of a separable superreflexive space is Haar null. Or an
 elementary approach is to exhibit a measure $m$ such that $m(x_{0}+
 E_{M})=0$ for every $x_{0}\in H$. We detail here the second
 argument.
The \mea s which we consider are  \nd\ \ga\ \mea s on $H$:
 let $(\Omega ,\mathcal{F},\P)$ be a standard probability space,
 and $(g_{n})_{n\geq 0}$ a sequence of standard independent
 random \ga\ variables, real or complex depending on whether the
 Hilbert space $H$ is supposed to be real or complex.
For any sequence $c=(c_{j})_{j\geq 0}$ of non-zero real numbers such that
$\sum_{j\geq 0}|c_{j}|^{2}<+\infty $, consider the random measurable
function $\Phi_{c}: (\Omega ,\mathcal{F},\P)\longrightarrow H $ defined
by
$$\Phi_{c}(\omega )=\sum_{j=0}^{+\infty }c_{j}g_{j}(\omega )f_{j}.$$
This function is well-defined almost everywhere, and it belongs to all
the spaces $L^{p}(\Omega )$, $p\geq 1$. To each such function $\Phi_{c}$ is
associated a \mea\ $m_{c}$ defined on $H$ by
$$m_{c}(A)=
\P(\{\omega  \in \Omega \textrm{ ; } \Phi_{c}(\omega )\in A\})$$
for every Borel subset $A$ of $H$. This is a \ga\ \mea, and since all the
$c_{j}$'s are non-zero, its support is the whole space.

\begin{proposition}\label{prop18}
For any vector $x_{0}\in H$ and any $M \geq 1$, set
$$B_{x_{0},M}=\{\omega \in \Omega  \textrm{  ; for infinitely many }
n's, |e_{0}^{*(n)}(x_{0}+\Phi_{c}(\omega ))|\geq 2^{-n} M
  \}.$$ Then $\P(B_{x_{0},M})=1$.
\end{proposition}

\begin{proof}
Write $x_{0}=\sum_{j=0}^{+\infty }u_{j}f_{j}$. Then $\pi_{[0,\xi
_{n}]}(x_{0}+\Phi_{c}(\omega ))
=\sum_{j=0}^{\xi _{n}}(u_{j}+c_{j}g_{j}(\omega ))f_{j}$
for every $n\geq 0$. Consider the random variable
$$X_{n}(\omega )=\pss {e_{0}^{*(n)}}
 {\pi_{[0,\xi _{n}]}(x_{0}+\Phi_{c}(\omega ))}
 =
 \sum_{j=0}^{\xi _{n}}\pss {e_{0}^{*(n)}}
 {(u_{j}+c_{j}g_{j}(\omega ))f_{j}}
,$$ where for $x=\sum_{j=0}^{\xi _{n}}\alpha _{j}^{(n)}e_{j}$, $\pss
{e_{0}^{*(n)}} x =\alpha _{0}^{(n)}$. Then $X_{n}$ is a \ga\ random
variable with mean $m_{n}=\sum_{j=0}^{\xi _{n}} u_{j} \pss {e_{0}^{*(n)}}
{f_{j}} $ and variance
$$\sigma _{n}=\sqrt{\sum_{j=0}^{\xi _{n}}|c_{j}|^{2}\,
|\pss {e_{0}^{*(n)}} {f_{j}} |^{2}}\geq |c_{0}|\, |\pss {e_{0}^{*(0)}}
{f_{0}} |=|c_{0}|.$$ Let us estimate $\P(|X
_{n}|\leq 2^{-n}M)$. If the space $H$ is real,
$$\P(|X
_{n}|\leq 2^{-n}M)=\int_{-2^{-n}M}^{2^{-n}M} \exp({-\frac {1}{2
\sigma _{n}^{2}}(t-m_{n})^{2}})\frac {1}{\sigma
_{n}\sqrt{2\pi}}dt
\leq \frac {2^{-(n-1)}M}{\sigma _{n}\sqrt{2\pi}}\leq
\frac {2^{-n}M}{|c_{0}|}\cdot$$ If the space $H$ is complex,
\begin{eqnarray*}
\P(|X
_{n}|\leq 2^{-n}M)&=&\int_{\sqrt{u^{2}+v^{2}}\leq 2^{-n}M}
\exp({-\frac {1}{2
\sigma _{n}^{2}}|u+iv-m_{n}|^{2}})\frac {1}{
\sigma _{n}^{2}.{2\pi}}dudv\\
&\leq& \frac {2^{-2n}M^{2}}{2\sigma _{n}^{2}}\leq
\frac {2^{-2n}M^{2}}{2|c_{0}|^{2}}\cdot
\end{eqnarray*}
In both cases the series $\sum_{n\geq 0}\P(|X_{n}|\leq 2^{-n}M)$ is
convergent. By the Borel-Cantelli Lemma, the probability that $|X_{n}
|\leq 2^{-n}M$ for infinitely many $n$'s is zero, and this is exactly the
statement of Proposition \ref{prop18}.
\end{proof}

Theorem \ref{th2} follows immediately from Propositions \ref{prop16} and
\ref{prop18}: for any $x_{0}\in H$,
\begin{eqnarray*}
m(-x_{0}+E_{M})&=&\P(\{\omega \in \Omega \textrm{ ; }
x_{0}+\Phi_{c}(\omega )\in E_{M}\})\\ &\leq&\P(\{\omega \in \Omega
\textrm{ ; } \exists n_{0}
\;\forall n\geq n_{0} \;|e_{0}^{*(n)}(x_{0}+
\Phi_{c}(\omega ))|\leq 2^{-n}M\})=0.
\end{eqnarray*}
Hence each set $E_{M}$ is Haar null.

\subsection{The set $HC(T)^{c}$ is $\sigma $-porous}
It is not difficult to see that $HC(T)^{c}$ is also $\sigma $-porous in
this example. Indeed let $\tilde{E}_{M}$ be the set of $ x \in H$ such
that $||x||<M $ and there exists an $n_{0}$ such that for all $ n\geq
n_{0}$, $|e_{0}^{*(n)}(x)|< 2^{-n}M$. Write
$\tilde{E}_{M}=\cup_{n_{0}\geq 1}\tilde{E}_{M,n_{0}}$ where
$\tilde{E}_{M,n_{0}}$ is the set of $x \in H $ such that $ ||x||<M  $ and
for every $ n\geq n_{0}$, $|e_{0}^{*(n)}(x)|< 2^{-n}M$. We are going to
show that each one of the sets $\tilde{E}_{M,n_{0}}$ is $\frac
{1}{2}$-porous. For each $n\geq 1$, let $x_{n}\in F_{\xi _{n}}$,
$||x_{n}||=1$, be such that $e^{*(n)}
_{0}(x_{n})=||e_{0}^{*(n)}||$. If we suppose for instance that $p_{1,n}
=1$ for every $n\geq 1$, then $f_{c_{1,n}}=\gamma _{n}^{-1}(e_{c_{1,n}}
-e_{0})$ so $e_{0}^{*(n)}(f_{c_{1,n}})=-\gamma _{n}^{-1}$ and
hence $||e_{0}^{*(n)}||\geq \gamma _{n}^{-1}$. Thus by choosing $\gamma
_{n}$ sufficiently small at each step, it is possible to ensure that
$||e_{0}^{*(n)}||\geq 2^{n}$ for every $n$. So given $x\in
\tilde{E}_{M,n_{0}}$ and $\varepsilon >0$, let $0<\delta <\varepsilon $
be so small that $||z||<M$ for every $z$ such that $||z-x||\leq \delta $.
Fix $k\geq n_{0}$ such that $\frac {1}{2} \delta
||e_{0}^{*(k)}||>2\cdot 2^{-k} \cdot M$ and choose $y=x+\delta x_{k}$. Then
$||y||<M$ and $0<||y-x||<\varepsilon $. Consider $z\in B(y, \frac
{1}{2}||y-x||)=B(y,\frac {\delta }{2})$. Then
\begin{eqnarray*}
|e_{0}^{*(k)}(z)|&\geq &|e_{0}^{*(k)}(y)|-||e_{0}^{*(k)}||\, ||z-y||\\
&\geq&\left|e_{0}^{*(k)}(x)-\delta ||e_{0}^{*(k)}|| \right|-\frac {\delta
}{2} ||e_{0}^{*(k)}||\\ &\geq&\frac {\delta
}{2}||e_{0}^{*(k)}||-2^{-k}M>2^{-k}M
\end{eqnarray*}
by our assumption on $k$. Hence $z\not \in \tilde{E}_{M,n_{0}}$, and
$B(y, \frac {1}{2}||y-x||)\cap \tilde{E}_{M,n_{0}}$ is empty. This proves
that $\tilde{E}_{M,n_{0}}$ is $\frac {1}{2}$-porous.

\section{Orbit-unicellularity of $T$: proofs of Theorems \ref{th2} and \ref{th2bis}}

The main step in the proof of Theorem \ref{th2} is Theorem \ref{th2bis},
which shows that whenever $x$ and $y$ are two vectors of $H$ of norm
$1$, either the closure of the orbit of $x$ is contained in the closure
of the orbit of $y$, or the other way round. In view of Proposition
\ref{prop6}, this is quite a natural statement: the idea of the proof of
Proposition \ref{prop6} is that whenever
$\pi_{[0,\xi _{n}]}x=\sum_{j=r_{n}}^{\xi _{n}}e_{j}^{*(n)}(x)e_{j}$ with
$e_{r_{n}}^{*(n)}(x)\not = 0$, then for every vector $z$ supported in $[r_{n},
\xi _{n}]$ there exists a polynomial $p$ of degree less than $\xi _{n}$
such that $p(T_{\xi _{n}})\pi_{[0,\xi _{n}]}x=z$, and $|p|$ is
controlled by a constant which depends on $|e_{r_{n}}^{*(n)}(x)|$ (and $\xi _{n}$
of course). If our two vectors $x$ and $y$ are given:

-- either there are
infinitely many $n$'s such that the first ``large''  $e_{j}$-coordinate (in a sense
to be made precise later) of $\pi_{[0,\xi _{n}]}x$ is smaller than
the first ``large''  $e_{j}$-coordinate of $\pi_{[0,\xi _{n}]}y$, and in this case
there exists infinitely many polynomials $p_{n}$ suitably controlled
such that $||p_{n}(T_{\xi _{n}})\pi_{[0,\xi _{n}]}x-\pi_{[0,\xi _{n}]}y||\leq \frac {1}{
\xi _{n}}$ for instance for these
$n$'s,

-- or the first large coordinate appears first in $\pi_{[0,\xi _{n}]}y$
infinitely many times, and then $||p_{n}(T_{\xi _{n}})
\pi_{[0,\xi _{n}]}y-\pi_{[0,\xi _{n}]}x||\leq\frac {1}{\xi _{n}}$.

\par\smallskip
In the first case $y$ will belong to the closure of the orbit of $x$,
and in the second case $x$ will belong to the closure of the orbit of
$y$.
\par\smallskip
In order to be able to formalise this argument, we have to quantify what
it means for an $e_{j}$-coordinate to be ``large'', and for this it will
be useful to have a precise estimate on $|p|$ for polynomials $p$ such
that $p(T_{\xi _{n}})\pi_{[0,\xi _{n}]}x=z$ as above in terms of the
size of $|e_{r_{n}}^{*(n)}(x)|$.

\begin{lemma}\label{lem1b}
For every $n
\geq 1$ there exists a constant $C_{\xi _{n}}'$ depending only on $\xi _{n}$
such that the following property holds true:
\par\smallskip
for every vector $x$ of $F_{\xi _{n}}$
of norm $1$, $x=\sum_{j=r_{n}}^{\xi _{n}}e_{j}^{*(n)}(x)e_{j}$ with
$e_{r_{n}}^{*(n)}(x)\not = 0$, and for every vector $y$ of norm $1$
belonging to the linear span of the vectors
$e_{r_{n}}, \dots, e_{\xi _{n}}$, there exists a
polynomial $p$ of degree less than $\xi _{n}$ with
$$|p|\leq \frac {C_{\xi _{n}}'}{
|e_{r_{n}}^{*(n)}(x)|^{\xi _{n}-r_{n}+1}}$$ such that
$p(T_{\xi_{n}})x=y$.
\end{lemma}

\begin{proof}
If $p({\zeta })=\sum_{u=0}^{\xi _{n}}a_{u}\zeta ^{u}$, then since
$T_{\xi _{n}}^{u}e_{j}=e_{j+u}$ for $j+u\leq \xi _{n}$ and
$T_{\xi _{n}}^{u}e_{j}=0$ for $j+u>\xi _{n}$, we have
\begin{eqnarray*}
p(T_{\xi _{n}})x&=&\sum_{u=0}^{\xi _{n}} a_{u}\sum_{j=r_{n}}^{\xi _{n}}
e_{j}^{*(n)}(x)e_{j+u}\\
&=&\sum_{j=r_{n}}^{\xi _{n}}e_{j}^{*(n)}(x)\sum_{u=j}^{\xi _{n}}
a_{u-j}e_{u}=\sum_{u=r_{n}}^{\xi _{n}} \left(
\sum_{j=r_{n}}^{u}e_{j}^{*(n)}(x)a_{u-j}\right)e_{u}.
\end{eqnarray*}
Hence solving the equation $p(T_{\xi _{n}})x=y$ boils down to solving
the system of $\xi _{n}-r_{n}+1$ equations
$$\sum_{j=r_{n}}^{u}e_{j}^{*(n)}(x)a_{u-j}=e_{j}^{*(n)}(y) \quad
\textrm{ for } u=r_{n},\dots, \xi _{n}.$$
This can be written in matrix form as
$$
\begin{pmatrix}
e_{r_{n}}^{*(n)}(x) & & (0)\\
e_{r_{n}+1}^{*(n)}(x)&e_{r_{n}}^{*(n)}(x) &  &\\
\vdots &\ddots &  &\\
e_{\xi _{n}}^{*(n)}(x) &\ldots&e_{r_{n}}^{*(n)}(x) & \\
\end{pmatrix}
\begin{pmatrix}
a_{0}\\
a_{1}\\
\vdots \\
a_{\xi _{n}-r_{n}} \\
\end{pmatrix}=
\begin{pmatrix}
e_{r_{n}}^{*(n)}(y)\\
e_{r_{n}+1}^{*(n)}(y)\\
\vdots \\
e_{\xi _{n}}^{*(n)}(y)\\
\end{pmatrix}
$$
and if $M_{\xi _{n}}(x)$ denotes the square matrix of size
$\xi _{n}-r_{n}+1$ on the left-hand side, then it is invertible. If we
choose
$$\begin{pmatrix}
a_{0}\\
a_{1}\\
\vdots \\
a_{\xi _{n}-r_{n}} \\
\end{pmatrix}=M_{\xi _{n}}(x)^{-1}
\begin{pmatrix}
e_{r_{n}}^{*(n)}(y)\\
e_{r_{n}+1}^{*(n)}(y)\\
\vdots \\
e_{\xi _{n}}^{*(n)}(y)\\
\end{pmatrix}$$
then $p(T_{\xi _{n}})x=y$.
Hence $$|p|\leq ||M_{\xi _{n}}(x)^{-1}||_{\mathcal{B}(\ell_{1})}
\sum_{j=r_{n}}^{\xi _{n}}|e_{j}^{*(n)}(y)|\leq ||M_{\xi _{n}}(x)^{-1}||_{\mathcal{B}(\ell_{1})}
A_{\xi _{n}}$$ since $||y||=1$, and
$$||M_{\xi _{n}}(x)^{-1}||_{\mathcal{B}(\ell_{1})}\leq \frac {B_{\xi _{n}}}{
|e_{r_{n}}^{*(n)}(x)|^{\xi _{n}-r_{n}+1}}$$ since $||x||=1$, which
proves Lemma \ref{lem1b}.
\end{proof}

Let now $x$ and $y$ be our two vectors of $H$ with $||x||=||y||=1$. We
will say that the $e_{j}$-coordinate of $\pi_{[0,\xi _{n}]}x$ is large
if $$|e_{j}^{*(n)}(x)|\geq \frac {1}{C_{\xi _{n}}^{\xi _{n}-j+1}}$$
where $C_{\xi _{n}}$ is a constant depending only on $\xi _{n}$ which
will be chosen later on in the proof.
A first point is:

\begin{fact}\label{fact2b}
Provided the sequence $(C_{\xi _{n}})$ grows fast enough, for every $x\in
H$, $||x||=1$, there exists an $n_{0}$ such that for every $n\geq
n_{0}$, there exists a $j\in [0,\xi _{n}]$ with
$$|e_{j}^{*(n)}(x)|\geq \frac {1}{C_{\xi _{n}}^{\xi _{n}-j+1}} \cdot$$
\end{fact}

\begin{proof}
Suppose on the contrary  that for every $j\in[0,\xi _{n}]$,
$|e_{j}^{*(n)}(x)|\leq {1}/{C_{\xi _{n}}^{\xi _{n}-j+1}} $. Then
$$||\pi_{[0,\xi _{n}]}x||\leq \sum_{j=0}^{\xi _{n}}|e_{j}^{*(n)}(x)|\, \sup_{0\leq j
\leq \xi _{n}}||e_{j}||\leq \frac {1}{C_{\xi _{n}}-1}
\, \sup_{0\leq j
\leq \xi _{n}}||e_{j}||.$$ If $\sqrt{C_{\xi _{n}}}\geq  \sup_{0\leq j
\leq \xi _{n}}||e_{j}|| $ for instance,
$||\pi_{[0,\xi _{n}]}x||\leq {\sqrt{C_{\xi _{n}}}}/({C_{\xi _{n}}-1})$,
and since $||x||=1$ this is impossible if $n$ is large enough and $C_{\xi_{n}}$
goes fast enough to infinity.
\end{proof}

\subsection{Proof of Theorem \ref{th2bis}}
Denote by $j_{n}(x)$ the smallest integer $j$ in $[0,\xi _{n}]$ such
that $|e_{j}^{*(n)}(x)|\geq {1}/{C_{\xi _{n}}^{\xi _{n}-j+1}} $. Then
either for infinitely many $n$'s $j_{n}(x)\leq j_{n}(y)$, or for
infinitely many $n$'s $j_{n}(y)\leq j_{n}(x)$. In the rest of the proof we suppose
that $j_{n}(x)\leq j_{n}(y)$ for infinitely many $n$'s and write $j_{n}=
j_{n}(x)$:
$$|e_{j_{n}}^{*(n)}(x)|\geq \frac{1}{C_{\xi _{n}}^{\xi _{n}-j_{n}+1}}
$$ and for every $j<j_{n}$,
$$|e_{j}^{*(n)}(x)|\leq \frac{1}{C_{\xi _{n}}^{\xi _{n}-j+1}}
\;\;\textrm{ and }\;\;
|e_{j}^{*(n)}(y)|\leq \frac{1}{C_{\xi _{n}}^{\xi _{n}-j+1}} \cdot$$
By Lemma \ref{lem1b} applied to the two vectors
$x'=\sum_{j=j_{n}}^{\xi _{n}}e_{j}^{*(n)}(x)e_{j}$
and $y'=\sum_{j=j_{n}}^{\xi _{n}}e_{j}^{*(n)}(y)e_{j}$, there exists a
polynomial $p_{n}$ of degree less than $\xi _{n}$ with
$|p_{n}|\leq C'_{\xi _{n}}. C_{\xi _{n}}^{\xi_{n}-j_{n}+1}$ such that
$$\left|\left|p_{n}(T_{\xi _{n}})\left(\sum_{j=j_{n}}^{\xi _{n}}e_{j}^{*(n)}(x)e_{j}
\right)-\sum_{j=j_{n}}^{\xi _{n}}e_{j}^{*(n)}(y)e_{j}\right|\right|\leq
\frac {1}{\xi_{n}}\cdot$$
Since
\begin{eqnarray*}
||\pi_{[0,\xi _{n}]}y-
\sum_{j=j_{n}}^{\xi _{n}}e_{j}^{*(n)}(y)e_{j}||&=&
||\sum_{j=0}^{j_{n}-1}e_{j}^{*(n)}(y)e_{j}||
\leq
\sum_{j=0}^{j_{n}-1}|e_{j}^{*(n)}(y)|\; \sup_{0\leq j<j_{n}} ||e_{j}||\\
&\leq& \frac {\sqrt{C_{\xi_{n}}}}{C_{\xi _{n}}-1}\leq \frac {1}{\xi _{n}}
\end{eqnarray*}
if $ \sup_{0\leq j
\leq \xi _{n}}||e_{j}|| \leq \sqrt{C_{\xi _{n}}} $ as above and
 $C_{\xi _{n}}$ grows fast enough, we get
$$\left|\left|p_{n}(T_{\xi _{n}})\left(\sum_{j=j_{n}}^{\xi _{n}}e_{j}^{*(n)}(x)e_{j}
\right)-\pi_{[0,\xi _{n}]}y\right|\right|\leq \frac {2}{\xi _{n}}\cdot$$
Then
\begin{eqnarray*}
||\pi_{[0,\xi _{n}]}y-p_{n}(T_{\xi _{n}})\left(
\sum_{j=0}^{j_{n}-1}e_{j}^{*(n)}(x)e_{j}\right)||&\leq&|p_{n}|2^{\xi _{n}}
\sum_{j=0}^{j_{n}-1}|e_{j}^{*(n)}(x)| \sqrt{C_{\xi_{n }}}\\
&\leq&
 {C'_{\xi _{n}}}{C_{\xi _{n}}^{\xi_{n}-j_{n}+1}}2^{\xi _{n}}
\sum_{j=0}^{j_{n}-1}\frac {1}{C_{\xi_{n}}^{\xi _{n}-j+1}}
 \sqrt{C_{\xi _{n}}}\\
&\leq&
 {C'_{\xi _{n}}}{C_{\xi _{n}}^{\xi_{n}-j_{n}+1}}2^{\xi _{n}}
\frac {2}{C_{\xi _{n}}^{\xi _{n}-j_{n}+2}} \sqrt{C_{\xi _{n}}}\\
&\leq&
\frac {C'_{\xi _{n}}2^{\xi _{n}}}{\sqrt{C_{\xi _{n}}}}\cdot
\end{eqnarray*}
Since $C_{\xi _{n}}$ can be chosen very large \wrt\ $C'_{\xi _{n}}$, we
can ensure that the quantity on the righthand side is less than $1/\xi
_{n}$, and hence
$$||p_{n}(T_{\xi _{n}})\pi_{[0,\xi _{n}]}x-\pi_{[0,\xi _{n}]}y||\leq \frac {3}{
\xi _{n}}\cdot$$
Now $|p_{n}|$ is controlled by a constant $D_{\xi _{n}}$
which depends only on $\xi _{n}$, and the same argument as in
Section $4$ (choosing $b_{n}$
very large \wrt\ $\xi _{n}$) shows that $$
||p_{n}(T)\pi_{[0,\xi _{n}]}x-\pi_{[0,\xi _{n}]}y||\leq \frac {4}{
\xi _{n}}\cdot$$ The polynomial $p_{n}$ has all the properties we want,
except for the fact that $\zeta $ does not necessarily divide $p_{n}(\zeta
)$, so consider $\tilde{p}_{n}(\zeta )=\zeta p_{n}(\zeta )$:
$$||\tilde{p}_{n}(T)\pi_{[0,\xi _{n}]}x-T(\pi_{[0,\xi _{n}]}y)||\leq \frac {8}{
\xi _{n}},$$
$|\tilde{p}_{n}|\leq D_{\xi _{n}}$ and the degree of $\tilde{p}_{n}$ is
less than $\xi _{n}+1$ (and not $\xi _{n}$ as before, but this is not a problem, as will
be seen shortly). We take as previously
$q_{n}(\zeta )=\frac {\zeta ^{b_{n}}}{b_{n}}\tilde{p}_{n}(\zeta )=
\frac {\zeta ^{b_{n}+1}}{b_{n}}{p}_{n}(\zeta )$: $|q_{n}|<1$ and the
degree of $q_{n}$ is less than $\nu _{n}=\xi _{n}(b_{n}+1)$. We have
$$||\left(\frac {T^{b_{n}}}{b_{n}}p_{n}(T)-p_{n}(T)\right)\pi_{
[0,\xi _{n}]}x||\leq \frac {C_{\xi _{n}}}{b_{n}}|p_{n}|2^{\xi _{n}}$$ by
Fact \ref{fact12}, so $$
||q_{n}(T)\pi_{[0,\xi _{n}]}x-\tilde{p}_{n}(T)\pi_{[0,\xi _{n}]}x||\leq
\frac {C_{\xi _{n}}}{b_{n}}|p_{n}|2^{\xi _{n}+1}\leq \frac {1}{\xi_{n}}$$
if $b_{n}$ is large enough.
Thus $$||q_{n}(T)\pi_{[0,\xi _{n}]}x-T(\pi_{[0,\xi n]}y)||\leq \frac {9}{\xi _{n}}
$$
and the proof then goes as in Proposition \ref{prop6}: for some $k\in
[1,k_{n}]$,
$$||T^{c_{k,n}}x-T(\pi_{[0,\xi _{n}]}y)||\leq \frac {10}{\xi _{n}}\cdot$$
Since $T(\pi_{[0,\xi _{n}]}y)$ tends to $Ty $ as $n$ tends to infinity,
this shows that $Ty$ belongs to the closure of the orbit of $x$. But
since $\overline{\mathcal{O}rb}(y,T)$ and $\overline{\textrm{sp}}[p(T)y \textrm{ ; }
p\in \K[\zeta ]]$ coincide, $y$ is a \hy\ vector for the \op\
induced by $T$ on $\overline{\textrm{sp}}[p(T)y \textrm{ ; }
p\in \K[\zeta ]]$, and thus $y$ is the limit of some sequence
$(T^{n_{j}}y)$. Hence $y\in \overline{\mathcal{O}rb}(x,T)$, which
proves that $\overline{\mathcal{O}rb}(y,T)\subseteq
\overline{\mathcal{O}rb}(x,T)$. This finishes the main part of the proof of Theorem
\ref{th2bis}.
\par\smallskip
We still have to prove that if $M$ is any non trivial invariant subspace
of $T$, the \op\ induced by $T$ on $M$ is \hy. Let $U$ and $V$ be two non
empty open subsets of $M$, with $u\in U$, $v\in V$. Since $T$ is
orbit-unicellular, either $\overline{\mathcal{O}rb}(u,T)\subseteq
\overline{\mathcal{O}rb}(v,T)$ or $\overline{\mathcal{O}rb}(v,T)\subseteq
\overline{\mathcal{O}rb}(u,T)$. Suppose for instance that we are in the
first case: $U$ and $V$ both intersect $\overline{\mathcal{O}rb}(v,T)\subseteq
M$, so there exist two integers $p$ and $q$, $q>p$, such that $T^{p}v\in U$
and $T^{q}v\in V$. Hence $T^{q-p}(U)\cap V$ is non empty. The same argument works if the
inclusion of the orbits of $u$ and $v$ is in the reverse direction, and this
proves that $T$ acting on $M$ is topologically transitive. The usual
Baire Category argument shows then that $T$ acting on $M$ is \hy, which
finishes the proof of Theorem \ref{th2bis}.

\subsection{Proof of Theorem \ref{th2}}
The proof of Theorem \ref{th2} is now easy, and follows the classical
argument which shows that any unicellular \op\ must be cyclic (see for instance
\cite{RR}): let $(x_{\alpha })_{\alpha \in A}$ be the family of all
non-\hy\ vectors for $T$, and for each $\alpha \in A$ write $M_{\alpha }=
\overline{\mathcal{O}rb}(x_{\alpha },T)$ (which is a closed nontrivial
subspace of $T$). So $$HC(T)^{c}=\bigcup_{\alpha \in A} M_{\alpha }$$ is a
linear subspace of $H$. If $HC(T)^{c}$ is not dense in $H$, then it is
contained in a closed hyperplane, and $HC(T)^{c}$ is clearly Gauss null.
So we can suppose that $HC(T)^{c}$ is dense in $H$. Then let $(x_{\alpha _{i}})_{
i\geq 0}$ be a countable subset of $HC(T)^{c}$ which is dense in $HC(T)^{c}$
(and hence in $H$). We are going to show that $$HC(T)^{c}=\bigcup_{i=0}^{+\infty }
M_{\alpha _{i}}.$$ Let $\alpha \in A$: we want to show that for some $i$,
$M_{\alpha }\subseteq  M_{\alpha _{i}}$. If this is not true, then by
Theorem \ref{th2bis} this means that $M_{\alpha _{i }}\subseteq M_{\alpha }$
for every $i$, hence $x_{\alpha _{i}}\in M_{\alpha }$ for every $i$.
Since $(x_{\alpha _{i}})_{i\geq 0}$ is dense in $H$, $M_{\alpha }=H$, so
$x_{\alpha }$ is \hy, a contradiction. Thus $HC(T)^{c}=\cup_{i=0}^{+\infty }M_{\alpha _{i}}$
is a countable union of subsets of closed hyperplanes, and hence is
Gauss null.

\subsection{Proof of Proposition \ref{prop2ter}}
Suppose that every \op\ on an infinite dimensional separable Hilbert
space has a non trivial invariant subspace, and let $T\in \mathcal{B}(H)$
satisfy the assumptions of condition (b) in Proposition \ref{prop2ter}.
It is not difficult to see that $HC(T)^{c}$ can be written as a strictly
increasing union of closures of orbits
$$HC(T)^{c}=\bigcup_{n\in\Z} M_{{n}} \qquad \textrm{ with } M_{{n}}
\subsetneq M_{{n+1}} \textrm{ for every } n\in\Z.$$  Indeed
consider the decomposition $HC(T)^{c}=\cup_{i=0}^{+\infty }M_{\alpha _{i}}$
obtained in the proof of Theorem \ref{th2}. Take $M_{{0}}=M_{\alpha
_{0}}$. Since $M_{\alpha _{1}}\not = H$ and the sequence $(x_{\alpha _{i}})$
is dense in $H$, there exists an $\alpha _{i}$ such that
 $M_{{0}}\subsetneq M_{\alpha _{i}}$. Take $i_{1}$ to be the smallest
 integer such that this property holds true, and set $M_{{1}}=M_{{\alpha
 _{i_{1}}}}$. In the same way let $j_{1}$ be the smallest integer such
 that $M_{\alpha _{j_{1}}}\subsetneq M_{0}$, and set $M_{-1}=M_{\alpha
 _{j_{1}}}$. In this fashion we construct two strictly increasing
 sequences $(i_{n})_{n\geq 1}$
 and $(j_{n})_{n\geq 1}$ of integers
having the property that for every $i<i_{n}$, $M_{\alpha _{i}}
 \subseteq M_{i_{n-1}}$ and $M_{i_{n-1}}\subsetneq M_{i_{n}}$, and
for every $j<j_{n}$, $M_{\alpha _{j_{n-1}}}
 \subseteq M_{j}$ and $M_{j_{n}}\subsetneq M_{j_{n-1}}$. Setting
 $M_{n}=M_{i_{n}}$  and $M_{-n}=M_{j_{n}}$ for $n\geq 1$, we get that
 this sequence of subspaces is strictly increasing, and that for every
$i\geq 0$ there exists an $n$ such that $M_{{-n}}
\subseteq M_{\alpha _{i}}\subseteq
M_{{n}}$. Hence $HC(T)^{c}=\cup_{n\in\Z} M_{n}$. For $n\in\Z$ set
$\Phi(n)=M_{n}$. Since $M_{n}\subsetneq M_{n+1}$, $M_{n+1}/M_{n}$ is non
trivial, and by the argument given in the introduction this quotient
must be a Hilbert space of infinite dimension. By our assumption, $T$
acting on $M_{n+1}/M_{n}$ has a non trivial invariant subspace. This
means that there exists $M$ invariant for $T$ such that
$M_{n}\subsetneq M
\subsetneq M_{n+1}$. Set $\Phi(n+1/2)=M$. Continuing in this
fashion, we can define in the obvious way the
subspaces $\Phi(n+\sum_{k\in I}2^{-k})$ where $I$
 is any finite subset of the set of positive integers.
 Clearly if $n_{1}+\sum_{k\in I_{1}}2^{-k}<
 n_{2}+\sum_{k\in I_{2}}2^{-k}$,
 $\Phi(n_{1}+\sum_{k\in I_{1}}2^{-k})
\subsetneq\Phi(n_{2}+\sum_{k\in I_{2}}2^{-k})$.
 We now wish to extend $\Phi$
 to an increasing and injective application from $\R$ into the set of
 invariant subspaces (or equivalently orbits) of $T$. For $t\in\R$ set
 $$\Phi(t)=\overline{\bigcup \Phi(n+\sum_{k\in I}2^{-k})},$$ where
 the union is taken over all the numbers of the form $n+\sum_{k\in I}
2^{-k}$
 which are less or equal to $t$. This is clearly an invariant subspace
 of $T$, and if $t\leq s$ obviously $\Phi(t)\subseteq \Phi(s)$. If
 $t<s$, there exist two numbers of the form
 $n+\sum_{k\in I}2^{-k}$ such that
 $$t<n_{1}+\sum_{k\in I_{1}}2^{-k}<n_{2}+
 \sum_{k\in I_{2}}2^{-k}<s.$$ Hence
 $\Phi(t)\subseteq
 \Phi(n_{1}+\sum_{k\in I_{1}}2^{-k})\subsetneq
 \Phi(n_{2}+
 \sum_{k\in I_{2}}2^{-k})\subseteq
 \Phi(s)$, and thus $\Phi$ is increasing and injective.

\section{Orbit-reflexive operators: proof of Theorem \ref{th3}}

Let $T$ be the \op\ constructed in Section $4$.
In order to show that $T$ is not orbit-reflexive, it suffices to exhibit
an \op\ $A$ which has the property that $Ax\in \overline{\mathcal{O}rb}(x,T)$
for every $x\in H$, but $A$ does not commute with $T$. The natural idea would be
to consider $A$ defined by $Ae_{0}=0$ and $Ae_{i}=e_{i+1}$ for $i\geq 1$.
Unfortunately this \op\ can be unbounded: suppose for instance that
$p_{1,n}=1$: $f_{c_{1,n}}=\gamma_{n} ^{-1} (e_{c_{1,n}}-e_{0})$,
$$Af_{c_{1,n}}=\gamma _{n}^{-1}e_{c_{1,n}+1}
=f_{c_{1,n}+1}+\gamma _{n}^{-1}e_{1}$$ and thus $A$ is unbounded.
A way to circumvent this difficulty is to modify the construction of $T$
and to take for the $p_{k,n}$'s
polynomials whose $0$-coefficient vanishes: let $(p_{k,n})
_{1\leq k\leq k_{n}}$ be a $4^{-\nu _{n}}$-net of the set of
polynomials $p$ of degree less than $\nu _{n}$ such that $|p|\leq 1$ and
$p(0)=0$. Then the definition of $f_{j}$ for $j\geq 1$ in the (b)- and
(c)-working intervals depends only on $e_{j}$ for $j\geq 1$. Since
$Ae_{j}=Te_{j}$ for $j\geq 1$, this yields that $ Af_{j}=Tf_{j}$ for
every $j\geq 1$. Hence

\begin{fact}\label{fact19}
The \op\ $A$ is bounded on $H$.
\end{fact}

Remark that with this choice of the polynomials $p_{k,n}$, $T$ is no
longer \hy. Clearly $A$ and $T$ do not commute, since $TAe_{0}=0$ while
$ATe_{0}=Ae_{1}=e_{2}$. Theorem \ref{th3} is a direct consequence of this
and the next proposition.

\begin{proposition}\label{prop20}
For every $x\in H$, $Ax $ belongs to the closure of the orbit of $x$
under the action of $T$.
\end{proposition}

\begin{proof}
$\bullet$ If $\pss x {e_{0}} =0$, i.e $x=\sum_{j=1}^{+\infty
}x_{j}f_{j}$, then $Ax=Tx$.
\par\smallskip
$\bullet$ If $\pss x {e_{0}} =\alpha  \not =0$, then for every $n\geq 1$,
$\pi_{[0,\xi _{n}]}x=\alpha e_{0}+\sum_{j=1}^{\xi _{n}}e_{j}^{*(n)}(x)
e_{j}$, so if $n$ is large enough $|e _{0}^{*(n)}(x)|\geq 2^{-n}$. Hence
assumption (*) of Proposition \ref{prop16} is satisfied. Using the
notation of the proof of Proposition \ref{prop16}, $$||q(T)\pi_{[0,\xi
_{n}]}x-e_{1}||\leq \frac {3}{
\xi _{n}},$$ where $q$ is of degree less than $\xi _{n}+b_{n}$, $|q|<1$
and $\zeta ^{b_{n}+1}$ divides $q(\zeta )$. In particular $\zeta $
divides $q(\zeta )$, so with our definition of the polynomials $p_{k,n}$,
there exists a $k\leq k_{n}$ such that $|q-p_{k,n}|\leq 4^{-\nu _{n}}$.
Then the proof of Proposition \ref{prop16} shows that
$$||T^{c_{k,n}}x-e_{1}||\leq \frac {7}{\xi _{n}},$$ and hence $e_{1}$
belongs to the closure of the orbit of $x$. But the orbit of $e_{1}$
under $T$ is the linear span $H_{0}$ of the vectors $f_{j}$, $j\geq 1$.
This implies that the closure of $\mathcal{O}rb(x,T)$ contains
 $H_{0}$. Since
$Ax$ belongs to $H_{0}$, $Ax$ belongs to
$\overline{\mathcal{O}rb}(x,T)$, and this finishes the proof of Proposition
\ref{prop20}.
\end{proof}


\begin{thebibliography}{99999}
\par\bigskip

\bibitem{B}
\textsc{F.~Bayart,}
\newblock Porosity and hypercyclic operators,
\newblock \emph{Proc. Amer. Math. Soc.}, \textbf{133}
(2005), pp 3309 -- 3316.

\bibitem{BG}
\textsc{F.~Bayart, S.~Grivaux,}
\newblock Frequently hypercyclic operators,
\newblock \emph{Trans. Amer. Math. Soc.},  \textbf{358}  (2006), pp 5083 -- 5117

\bibitem{BMM}
\textsc{F.~Bayart, \'{E}.~Matheron, P.~Moreau,}
\newblock Small sets and hypercyclic vectors, preprint $2007$.

\bibitem{BL}
\textsc{Y.~Benyamini, J.~Lindenstrauss,}
\newblock Geometric Nonlinear Functional Analysis, Vol. $1$,
\newblock \emph{AMS Colloquium Publications}, Vol. $48$ (2000).


\bibitem{Bo}
\textsc{P.~Bourdon},
\newblock Invariant manifolds of hypercyclic vectors,
\newblock  \emph{Proc. Amer. Math. Soc.},  \textbf{118}  (1993), pp 845 -- 847.



\bibitem{E}
\textsc{P.~Enflo,}
\newblock On the Invariant Subspace Problem,
\newblock \emph{Acta Math.}, \textbf{158} (1987), pp 212 -- 313.

\bibitem{F}
\textsc{E.~Flytzanis},
\newblock Unimodular eigenvalues and linear chaos in Hilbert spaces,
\newblock \emph{Geom. Funct. Anal.},  \textbf{5}  (1995), pp 1 -- 13.

\bibitem{HRR}
\textsc{D.~Hadwin, E.~Nordgren, H.~Radjavi, P.~Rosenthal,}
\newblock Orbit-reflexive operators,
\newblock \emph{J. Lond. Math. Soc.}, \textbf{34} (1986),
pp 111 -- 119.

\bibitem{LM}
\textsc{F.~Le\'{o}n-Saavedra, A.~Montes-Rodr\'{\i}guez,}
\newblock Linear structure of hypercyclic vectors,
\newblock \emph{J. Funct. Anal.}, \textbf{148} (1997), pp 524 -- 545.


\bibitem{Ma}
\textsc{E.~Matouskova,}
\newblock Convexity and Haar null sets,
\newblock \emph{Proc. Amer. Math. Soc.}, \textbf{125} (1997), pp
1793 -- 1799.

\bibitem{M}
\textsc{V.~M\"uller,}
\newblock Mini-Workshop \emph{Hypercyclicity and Linear Chaos},
organized by T. Bermudez, G. Godefroy, K.-G. Grosse-Erdmann and A. Peris,
\newblock \emph{Oberwolfach Reports}, \textbf{3} (2006),
pp. 2227 -- 2276.

\bibitem{MV}
\textsc{V.~M\"uller, J.~Vrov\v{s}k\'y,}
\newblock On orbit-reflexive operators, preprint $2007$.


\bibitem{P}
\textsc{D.~Preiss,}
\newblock personal communication.

\bibitem{PT}
\textsc{D.~Preiss, J.~Tiser,}
\newblock Two unexpected examples concerning differentiability
of Lipschitz functions on Banach spaces,
\newblock \emph{GAFA Israel Seminar $1992-94$} (V. Milman and
J. Lindenstrauss editors), Birkha\"user (1995), pp 219 -- 238.

\bibitem{RR}
\textsc{H.~Radjavi, P.~Rosenthal,}
\newblock Invariant subspaces,
\emph{Ergebnisse der Mathematik and ihrer Grenzgebiete}, Band $77$,
Springer Verlag (1973).


\bibitem{R}
\textsc{C.~Read,}
\newblock A short proof concerning the Invariant Subspace Problem,
\newblock \emph{J. London Math. Soc.}, \textbf{34} (1986), pp 335 --
348.

\bibitem{R2}
\textsc{C.~Read,}
\newblock The Invariant Subspace Problem for a class of Banach spaces,
$2$: hypercyclic operators,
\newblock \emph{Israel J. Math.}, \textbf{63} (1988), pp 1 --
40.

\bibitem{R3}
\textsc{C.~Read,}
\newblock The invariant subspace problem on some Banach spaces with separable dual,
\newblock \emph{Proc. London Math. Soc.}, \textbf{58}  (1989), pp 583 -- 607.

\bibitem{Z}
\textsc{L.~Zaj\'{\i}\v{c}ek,}
\newblock Porosity and $\sigma $-porosity,
\newblock \emph{Real Anal. Exchange}, \textbf{13} (1987-88), pp
314 -- 350.
\end{thebibliography}
\end{document}